\documentclass[12pt, reqno]{amsart}
\pdfoutput=1
\usepackage{amsfonts}
\usepackage[margin=1in]{geometry}
\usepackage{amsmath,amssymb}
\usepackage{amsthm,enumerate}
\usepackage{graphicx,caption,color}
\usepackage[mathscr]{euscript}
\usepackage{float}
\usepackage{hyperref}
\usepackage{tabularx}

\renewcommand{\phi}{\varphi}

\renewcommand{\epsilon}{\varepsilon}
\hypersetup
{
    colorlinks,	
    citecolor=red,
    filecolor=black,
    linkcolor=blue,
    urlcolor=blue
}
\providecommand{\U}[1]{\protect\rule{.1in}{.1in}}
\newcounter{theorem}
\numberwithin{theorem}{section}
\newtheorem{thm}[theorem]{Theorem}
\newtheorem{conj}[theorem]{Conjecture}
\newtheorem{cor}[theorem]{Corollary}
\newtheorem{prop}[theorem]{Proposition}
\newtheorem{lem}[theorem]{Lemma}

\theoremstyle{definition}
\newtheorem{defn}[theorem]{Definition}
\newtheorem{expl}[theorem]{Example}
\newtheorem{rem}[theorem]{Remark}

\newcommand{\N}{\mathbb{N}}

\newcommand{\cU}{\mathcal{U}}
\newcommand{\x}[1]{x^{\underline{#1}}}

\newtheorem{letterthm}{Theorem}

\numberwithin{equation}{section}

\begin{document}
\title[Interpolation polynomials, bar monomials, and positivity]{%
Interpolation polynomials, bar monomials, and their positivity}
\author{Yusra Naqvi}
\address{Department of Mathematics, University College London, Gower Street, London WC1E 6BT, UK}
\email{y.naqvi@ucl.ac.uk}
\author{Siddhartha Sahi}
\address{Department of Mathematics, Rutgers University, 110 Frelinghuysen Rd, Piscataway, NJ 08854-8019, USA}
\email{sahi@math.rutgers.edu}
\author{Emily Sergel}
\address{Department of Mathematics, Rutgers University, 110 Frelinghuysen Rd, Piscataway, NJ 08854-8019, USA}
\email{esergel@math.rutgers.edu}
\thanks{Y.N. was partially supported by ARC grant DP180102437. S.S. was partially supported by NSF grants DMS-1939600 and 2001537, and Simons grant 509766. E.S. was partially supported by NSF grant DMS-1603681.}

\begin{abstract}
We prove a positivity result for interpolation polynomials that was conjectured by Knop and Sahi. These polynomials were first introduced by Sahi in the context of the Capelli eigenvalue problem for Jordan algebras, and were later shown to be related to Jack polynomials by Knop-Sahi and Okounkov-Olshanski. The positivity result proved here is an inhomogeneous generalization of Macdonald's positivity conjecture for Jack polynomials. We also formulate and prove the non-symmetric version of the Knop-Sahi conjecture, and in fact we deduce everything from an even stronger positivity result. This last result concerns certain inhomogeneous analogues of ordinary monomials that we call \emph{bar monomials}. Their positivity involves in an essential way a new partial order on compositions that we call the \emph{bar order}, and a new operation that we call a \emph{glissade}.
\bigskip



\end{abstract}

\maketitle
\vspace{-0.3in}

\section{Introduction}
\subsection{Main results}
The interpolation polynomials $P_{\lambda }^{\rho }(x)$ are inhomogenous symmetric polynomials in $x=(x_{1},\ldots ,x_{n})$ that were introduced by Sahi \cite{S1} following earlier work with Kostant \cite{KS1, KS2}, and are characterized by simple vanishing conditions described in \S\ref{sec:sym}. They are indexed by partitions $\lambda\in\N^n$, have degree $|\lambda|=\lambda_1+\cdots +\lambda_n$, and their coefficients depend on $n$ parameters $\rho=(\rho _{1},\ldots,\rho _{n})$. Of particular interest is the one-parameter family $\rho= r\delta, \delta=(n-1,\ldots,0)$ studied by Knop-Sahi \cite{KnS1} and Okounkov-Olshanski \cite{OO}.

The $P_{\lambda }^{r\delta}$ have a rich combinatorial structure that belies their simple definition. As shown in \cite{KnS1}, the top degree part of $P_{\lambda }^{r\delta}$ is the Jack polynomial $P_{\lambda}^{(\alpha )}$ with parameter $\alpha =1/r$. In his remarkable book, Macdonald \cite[VI.10.26?]{M} introduced a normalization  $J_{\lambda}^{(\alpha )}=c_{\lambda }(\alpha )P_{\lambda }^{(\alpha )}$ and conjectured that its coefficients lie in $\N[\alpha]$. This was proved by Knop and Sahi  \cite{KnS2}, who also gave a combinatorial formula for $J_{\lambda}^{(\alpha )}$ in terms of certain \emph{admissible} tableaux.

In this paper we extend the results of \cite{KnS2} to all of $P_{\lambda }^{r\delta}$. This involves the normalized polynomial $J_{\lambda }^{r\delta}(x)=(-1)^{|\lambda |}c_{\lambda }(\alpha )P_{\lambda
}^{r\delta }(-x)$ and its symmetric monomial expansion 
\begin{equation*}
J_{\lambda }^{r\delta }=\textstyle{\sum\nolimits_{\mu }}\alpha ^{|\mu |-|\lambda
|}a_{\lambda ,\mu }(\alpha )m_{\mu }.
\end{equation*}
We prove the following result conjectured by Knop and Sahi \cite[Conjecture 7]{KnS1}.

\begin{letterthm}
\label{thm:A} The coefficient $a_{\lambda ,\mu }(\alpha )$ is a polynomial in $\N[\alpha]$.
\end{letterthm}
The interpolation polynomials have nonsymmetric analogues $E_{\eta }^{\rho }$  \cite{Kn1, Sa2, Sa3} indexed by compositions $\eta \in \N^n$ and characterized by vanishing conditions described in \S\ref{sec:nonsym}. For $\rho=r\delta$, the top degree part of $E_{\eta}^{r\delta}$ is the nonsymmetric Jack polynomial $E_{\eta}^{(\alpha)}$ of Heckman and Opdam \cite{O}. After an explicit normalization, $F_{\eta }^{(\alpha )}=d_{\eta }(\alpha) E_{\eta }^{(\alpha )}$ has coefficients in $\N[\alpha]$. This was also proved in \cite{KnS2} and we now extend this to $E_{\eta}^{r\delta}$. More precisely, we consider the normalized polynomial $F_{\eta }^{r\delta }=(-1)^{|\eta |}d_{\eta }(\alpha )E_{\eta }^{r\delta}(-x)$ and its (ordinary) monomial expansion 
\begin{equation*}
F_{\eta }^{r\delta }=\textstyle{\sum\nolimits_{\gamma }}\alpha ^{|\gamma |-|\eta
|}b_{\eta ,\gamma }(\alpha )x^{\gamma }.
\end{equation*}

\begin{letterthm}
\label{thm:B} The coefficient $b_{\eta ,\gamma }(\alpha )$ is a polynomial in $\N[\alpha ]$.
\end{letterthm}

The homogeneous $F_{\eta }^{(\alpha )}$ and the inhomogeneous $F_{\eta }^{r\delta }$ are both linear bases for $\mathbb{F}[x_{1},\ldots ,x_{n}]$ over the field $\mathbb{F}=\mathbb{Q}(\alpha )=\mathbb{Q}(r)$. Thus there is a unique $\mathbb{F}$-linear ``dehomogenization'' operator $\Xi $ such that $\Xi (F_{\eta }^{(\alpha )})=F_{\eta }^{r\delta }$ for all $\eta \in \N^{n}$. Its action on monomials has the form
$$
\Xi(x^{\eta })=x^{\eta }+\textstyle{\sum_{|\gamma|<|\eta|}}c_{\eta,\gamma }(r)x^{\gamma },
$$
and we prove the following positivity result for $c_{\eta,\gamma }(r)$, which implies Theorems \ref{thm:A} and \ref{thm:B}.
\begin{letterthm}
\label{thm:C} The coefficient $c_{\eta ,\gamma }(r)$ is a polynomial in $%
\N[r]$ of degree $\leq \left\vert \eta \right\vert -\left\vert
\gamma \right\vert $.
\end{letterthm}

We write $x^{\underline{\eta }}=\Xi (x^{\eta })$ and refer to it as a \emph{bar monomial}. The notation is motivated by the fact that for $n=1$, we get the rising factorial $x^{\underline{k}}=x(x+1)\cdots(x+k-1)$. 

In view of Theorem \ref{thm:C}, it is natural to ask for a combinatorial formula for bar monomials that is manifestly positive and integral. We provide such a formula, which involves the following simple operation on the (English) Ferrers diagram of a composition:  

\vspace{6pt}
\begin{center}
\parbox{30em}{\tt Delete the last box from the highest row $k$ of maximal length $m$; then move $l$ $\ge$ $0$ boxes from the end of row $k$ to the end of another row, either above and strictly left, or below and weakly left of their original positions. }
\end{center}

\vspace{6pt}
We call this a \emph{glissade}, which in mountaineering means \textquotedblleft descent via a controlled slide\textquotedblright. We define the \emph{weight} of a glissade applied to $\gamma$ to be $r$ if $l>0$, otherwise we define it to be
$$ x_k+(m-1)+r\left( n-1-l_\gamma(k,m)\right).$$
Here $l_\gamma(k,m)$ is the \emph{leg} of the box $(k,m)$ in $\gamma$, which was defined in \cite{KnS2} as follows:
$$ l_{\gamma }\left( k,m\right)\ 
:=\ \#\left\{ i>k:m\leq \gamma _{i}\leq \gamma _{k}\right\}\ +\ \#\left\{ i<k:m\leq \gamma _{i}+1\leq \gamma _{k}\right\}.$$

If we start with some $\eta$ and apply a sequence of $|\eta|$ glissades then we necessarily arrive at $0$. We call such a sequence $G$ a \emph{bar game} on $\eta$, and we define its weight $w(G)$ to be the product of the weights of its glissades. We write $\mathcal{G}(\eta)$ for the set of all bar games on $\eta$, and we prove the following result that implies Theorem \ref{thm:C}, and hence also Theorems \ref{thm:A} and \ref{thm:B}.

\begin{letterthm}
\label{thm:D}We have $x^{\underline{\eta }}=\sum_{G\in \mathcal{G}(\eta)}w(G)$.
\end{letterthm}

\subsection{Examples.} \label{subsec:introex} \hphantom{a}
Before discussing the proof of Theorem \ref{thm:D}, we give three small examples to illustrate the various concepts. More detailed examples can be found in \S \ref{Many-Examples}.

\begin{figure}[h!]
\includegraphics[width=0.4\textwidth]{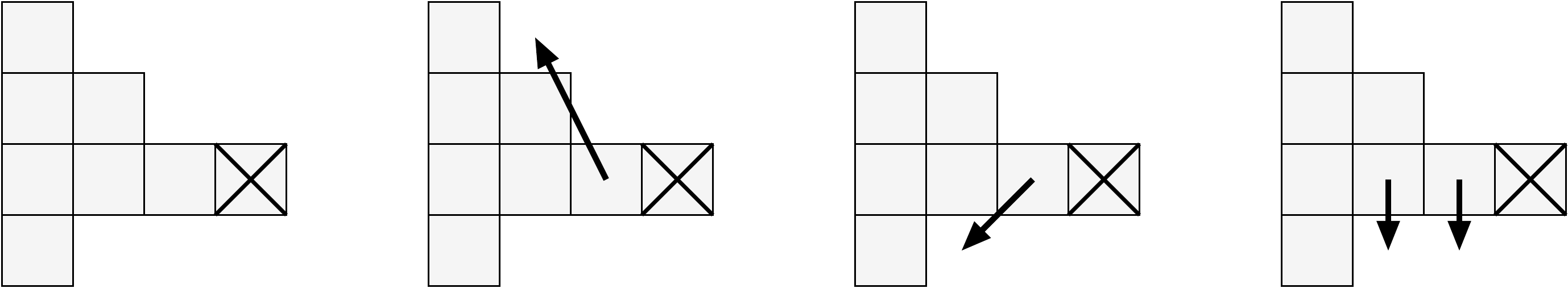}
\caption{All possible glissades on (1,2,4,1)}
\label{fig:1241moves}
\end{figure}

Figure \ref{fig:1241moves} shows all possible glissades on (1,2,4,1). The deleted box is indicated with a $\times$, and the arrows show the movement of other boxes. The resulting shapes are (1,2,3,1), (2,2,2,1), (1,2,2,2), and (1,2,1,3). See also Figure \ref{fig:Precur} for all moves on (1,4,1,2) and on (1,1,4,2).
\smallskip

\begin{figure}[h!]
\includegraphics[width=0.6\textwidth]{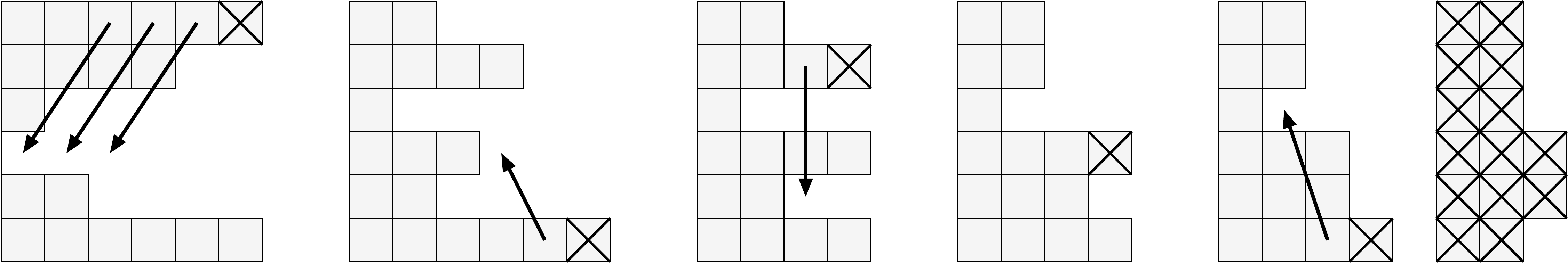}
\caption{A bar game on (6,4,1,0,2,6)}
\label{fig:introex}
\end{figure}

Figure \ref{fig:introex} shows a complete bar game on (6,4,1,0,2,6).  For the sake of space, when a box is deleted but no other boxes are moved, we put a $\times$ in that box and continue working with the same diagram. Thus the last diagram represents 14 deletions. This game has weight
$$ r^3 \cdot (x_4+3+4r) \cdot r \cdot (x_4+2+r) \cdot (x_5+2+r) \cdot \textstyle{\prod\nolimits_{k=1}^6}(x_k+1) \cdot \textstyle{\prod\nolimits_{k=1}^6}x_k $$

\begin{figure}[h]
\includegraphics[width=0.5\textwidth]{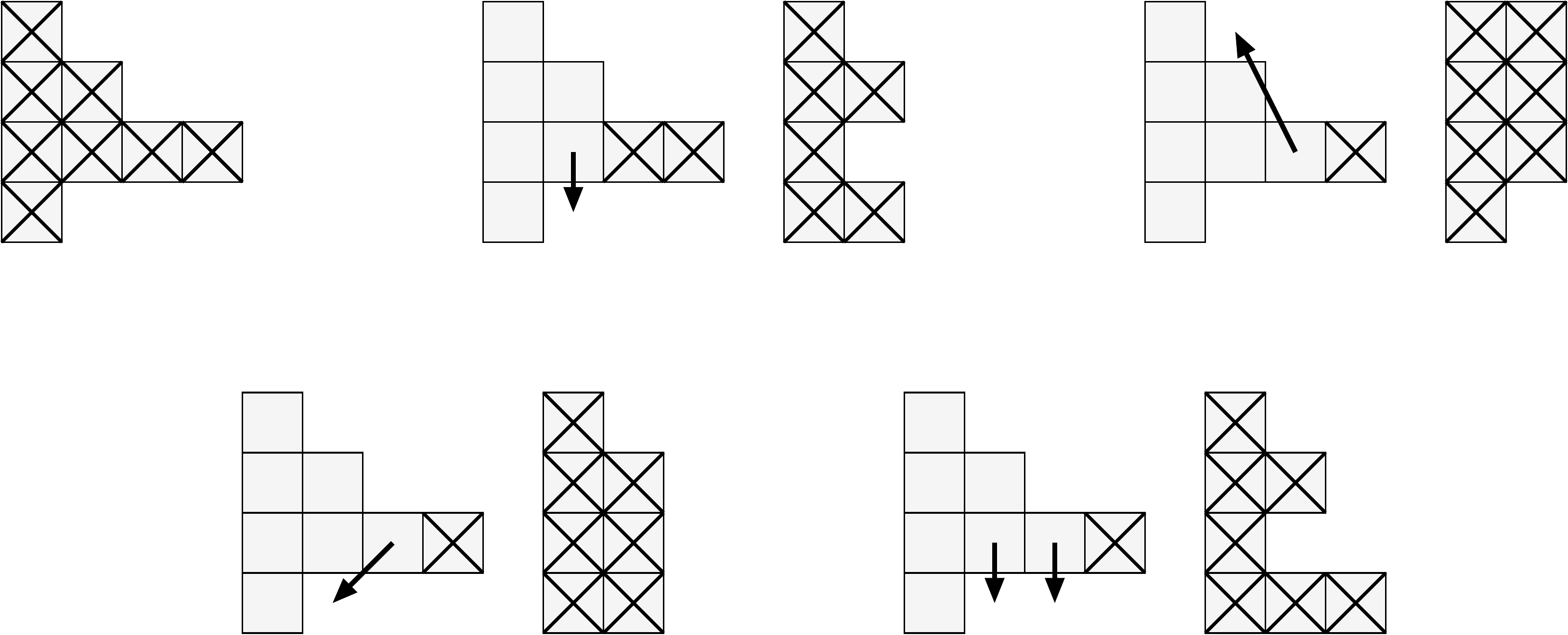}
\caption{All bar games on (1,2,4,1)}
\label{fig:1241games}
\end{figure}

Figure \ref{fig:1241games} shows all possible games on (1,2,4,1). There are 5 games in total, and taking their weighted sum gives the bar monomial $x^{\underline{(1,2,4,1)}}$. The explicit formula is given in \S \ref{subsec:examples}.

\subsection{Discussion of the proof.}

In sections \S \ref{sec:sym} and \ref{sec:nonsym} we recall the precise definitions of symmetric and nonsymmetric interpolation polynomials and their relationship with Jack polynomials. The symmetric polynomials are more natural objects, but it is easier to work with the nonsymmetric polynomials because they satisfy a recursion with respect to the graded affine Hecke algebra of the symmetric group \cite{Kn1, Sa2, Sa3}. This recursion is discussed in \S \ref{subsec:rec}; it is an inhomogeneous extension of a homogenous recursion that plays a key role in the proof of positivity for Jack polynomials \cite{KnS2}. However the inhomogeneous recursion does \emph{not }preserve positivity. This is the main reason why Theorems \ref{thm:A} and \ref{thm:B} remained conjectures for almost 25 years. 

In \S \ref{dehom} we introduce the dehomogenization operator, and use this to define the bar monomials in \S \ref{bar-mon}. In \S \ref{subsec:PfAB} we show how to deduce Theorems \ref{thm:A} and \ref{thm:B} from the positivity of bar monomials, {\it i.e.} from Theorem \ref{thm:C}. The bar monomials satisfy a recursion described in \S \ref{bar-mon}; this is simpler than the recursion of \S \ref{subsec:rec}, but it too is not positive.

The essential new results of the paper are in \S \ref{bar-games}. In \S \ref{crit-box} and \ref{gliss} we define the notion of a glissade and establish its properties under the action of the affine symmetric group. This is naturally related to a new partial order on compositions that we call the \emph{bar order}. In \S \ref{PfC} we define notion of a bar game, and show how to deduce Theorem \ref{thm:C} from Theorem \ref{thm:D}. In \S\ref{PfD} we prove Theorem \ref{thm:D}. The key here is the transition formula for bar monomials in Theorem \ref{thm:E}. This is proved using the recursions for bar monomials from \S \ref{bar-mon}, and it implies Theorem \ref{thm:D} by a simple iteration. Thus Theorem \ref{thm:D} can be regarded as a positive combinatorial solution to a non-positive recursion. 

We conclude the paper with some further examples illustrating Theorem \ref{thm:D}, and also explain how to use Theorem \ref{thm:D} to obtain combinatorial formulas for interpolation polynomials.
 
\subsection{Related results and open problems.}

Jack polynomials were introduced by Henry Jack \cite{J} as a one parameter generalization of Schur functions and of the zonal polynomials which play an important role in multivariate statistics \cite{Stat-Herz, Stat-Muir}. Along with Hall-Littlewood polynomials, they were one of the two key sources of inspiration for Macdonald's introduction of his two parameter family of symmetric functions \cite{M}; see \cite{Sahi-Book} for a historical background. These polynomials, in turn, were the impetus behind Cherednik's discovery of the double affine Hecke algebra \cite{Cher-DAHA-Mac, Cher-Book, Macd-DAHA, Sa-Koorn}.

Since their discovery, Jack polynomials and Macdonald polynomials have found an incredible number of applications in many different areas of mathematics. It is impossible to give anything approaching a complete accounting, but a partial list includes probability and statistics \cite{Process-BC, Process-BO, Process-OO, Stat-Rich}, harmonic analysis \cite{Ker-BF, Ker-R}, combinatorics \cite{Comb-GH, Comb-GH2, Comb-Hag-Book, Geom-Hai}, representation theory \cite{RT-Ion1,RT-Ion2}, algebraic geometry \cite{Geom-HLRV,Geom-HRV, Geom-Nak, Geom-RW, Geom-SV}, and knot theory \cite{BG,Knot-Cher}.

Symmetric Jack polynomials admit a formula in terms of semistandard tableaux \cite{M,S}, which generalizes the formula for Schur functions. However this involves weights that are rational functions in $\alpha$; thus it does not imply the integrality and positivity, which was conjectured by Macdonald, and which is immediate from the Knop-Sahi formula \cite{KnS1} in terms of admissible tableaux. The semistandard tableau formula has been generalized by Okounkov \cite{Ok1,OO} to interpolation polynomials, but it likewise does not imply Theorem \ref{thm:A}. Moreover there does not seem to be a nonsymmetric analog of Okounkov's formula.

As explained in \cite{S1}, interpolation polynomials arise naturally as solutions to the Capelli eigenvalue problem for invariant differential operators on a symmetric cone. The Capelli problem has analogues for other symmetric spaces studied in \cite{S-CapGrass, SS-Quad, SZ-Rad, SZ-Shim} and also for symmetric superspaces \cite{ASS, SS-Adv,SSS1}. The solutions of these other problems are related to interpolation polynomials defined by Okounkov, Ivanov, and Sergeev-Veselov \cite{I, Ok-BC, SV}. It would be interesting to see whether these classes of polynomials also have combinatorial interpretations along the lines of the present paper. 

Special values of interpolation polynomials appear as expansion coefficients at $x=\bf{1}$ in the binomial formula for Jack polynomials \cite{OO,Sa3}. These too seem to have a subtle positivity property, and it has been conjectured in \cite{Sa} that $ (-r)^{|\lambda|}J_\lambda^{r\delta}(-\mu-r\delta)$ belongs to $\N[\alpha]$ for all partitions $\lambda$ and $\mu$. Although this conjecture does not follow from the results of the present paper, the combinatorial ideas introduced do provide another line of attack. This is discussed further in section \ref{subsec:vanprop} below.

Interpolation analogues of symmetric and nonsymmetric Macdonald polynomials have been defined in \cite{Kn1, Sa2, Sa3}; these depend on two parameters $q$ and $t$. Thus it is natural to ask for a two parameter extension of the results of the present paper to the Macdonald setting. This will naturally involve a recursion with respect to the affine Hecke algebra, rather than its graded version. However it also seems to require additional ideas, and therefore we postpone this question to a subsequent paper. 

There has been considerable interest in Macdonald polynomials and interpolation polynomials in connection with integrable probability and solvable lattice models. In particular, the papers \cite{ABW, BW} describe formulas for Macdonald polynomials and related polynomials in terms of 6-vertex models. It is an open problem whether these formulas can be extended to the setting of interpolation polynomials. Relating the combinatorics of bar monomials to lattice models might offer some clues in this direction. 
 
For the special case $q=t$, the interpolation analogues of Macdonald polynomials are Harish-Chandra images of Capelli elements in the center of $\cU_q(\mathfrak{gl}_N)$. These central elements play a key role in the recent work of Beliakova and Gorsky \cite{BG}, which proves that the so-called ``universal link invariant'' dominates the Witten-Reshetekhin-Turaev invariants for $\cU_q(\mathfrak{gl}_N)$.  This work also raises the interesting problem of \emph{categorifying} the two parameter interpolation polynomials, with the expectation that this should have some applications to the study of knot and link invariants; see \cite{BG,GW} and the references therein. Perhaps the results of the present paper and its eventual extension to Macdonald polynomials might shed some light on this important question.

\section{Preliminaries}\label{sec:prel}

\subsection{Symmetric polynomials}\label{sec:sym}

The interpolation polynomials $P_{\lambda}^{\rho}\left(x\right)$  are inhomogeneous symmetric
polynomials that were introduced by Sahi in \cite{S1} following
earlier work with Kostant on generalizations of the Capelli identity \cite{KS1, KS2}. They are indexed by partitions 
\begin{equation*}
\mathscr{P}_{n}=\left\{ \lambda \in \mathbb{Z}^{n}\mid \lambda _{1}\geq
\cdots \geq \lambda _{n}\geq 0\right\} ,
\end{equation*}%
and their coefficients depend on $n$ indeterminates $\rho =\left( \rho
_{1},\ldots ,\rho _{n}\right) $.

\begin{thm}[\protect\cite{S1}]
\label{thm:Pdef}There is a unique symmetric polynomial $P_{\lambda }^{\rho
}\left( x\right) =P_{\lambda }^{\rho }(x_{1},\ldots ,x_{n})$ of total degree 
$\left\vert \lambda \right\vert =\lambda _{1}+\lambda _{2}+\cdots +\lambda
_{n}$ such that

\begin{enumerate}
\item $P_{\lambda }^{\rho }\left( \mu +\rho \right) =0$ \ for all $\mu \in %
\mathscr{P}_{n}$ with $\left\vert \mu \right\vert \leq |\lambda |,\mu
\neq \lambda ,$ and

\item the coefficient of the symmetric monomial $m_{\lambda }$ in $%
P_{\lambda }^{\rho }$ is $1$.
\end{enumerate}
\end{thm}

As explained in \cite{S1} the existence and uniqueness of these polynomials is equivalent to the following interpolation result.

\begin{thm}[\protect\cite{S1}]
\label{S-inter}A symmetric polynomial of degree $d$ is uniquely
characterized by its values on the set $\left\{ \mu +\rho :\left\vert \mu
\right\vert \leq d\right\} .$
\end{thm}

The case $\rho =r\delta $ with $\delta =\left( n-1,\ldots ,1,0\right) $ was studied in some detail by Knop and Sahi \cite{KnS1}, and is related to Jack polynomials  $P_{\lambda }^{(\alpha )}$ with parameter $ \alpha =1/r$ \cite{M,S}.

\begin{thm}[\protect\cite{KnS1}]
We have $P_{\lambda }^{r\delta }=P_{\lambda }^{(\alpha )}+$  terms of degree $ <|\lambda |.$
\end{thm}

For a box $s=\left( i,j\right)$ in the Ferrers diagram of $\lambda$, its  \emph{arm} and \emph{leg} are defined to be
\begin{equation*}
a_{\lambda }\left( i,j\right) =\lambda _{i}-j,\quad l_{\lambda }\left(
i,j\right) =\#\left\{ k>i:\lambda _{k}\geq j\right\}.
\end{equation*}
We set $c_{\lambda }\left( \alpha \right) =\prod\nolimits_{s\in \lambda }\left(
\alpha a_{\lambda }\left( s\right) +l_{\lambda }\left( s\right) +1\right)$ and we define the normalized Jack polynomial to be
\begin{equation*}
J_{\lambda }^{(\alpha )}=c_{\lambda }\left( \alpha \right) P_{\lambda}^{(\alpha )}
\end{equation*}%

\begin{thm}[\protect\cite{KnS2}]
\label{thm:macd-conj}The coefficients of $J_{\lambda }^{(\alpha )}$ with
respect to the $m_{\mu }$ belong to $\mathbb{N}[\alpha ]$.
\end{thm}

This was conjectured by Macdonald in his book \cite[VI.10.26?]{M}. 
The paper \cite{KnS2} also provides a combinatorial formula for $J_{\lambda }^{(\alpha )}$ in terms of certain \textquotedblleft
admissible\textquotedblright\ tableaux.

In \cite{KnS1}, Knop and Sahi introduced a normalized
version of the interpolation polynomial, which involves the same constant $%
c_{\lambda }(\alpha )$ together with a sign twist. They also made a
conjecture concerning its expansion coefficients with respect to $m_{\mu }$,
which generalizes Macdonald's conjecture (Theorem \ref{thm:macd-conj}).

\begin{defn}
\label{defn: symm-expansion}The normalized symmetric interpolation
polynomial is 
\begin{equation}
J_{\lambda }^{r\delta }:=(-1)^{|\lambda |} \, c_{\lambda }(\alpha ) \, P_{\lambda}^{r\delta }(-x),  \label{=Jrd}
\end{equation}%
and its expansion coefficients $a_{\lambda ,\mu }\left( \alpha \right) $ are
defined by%
\begin{equation}
J_{\lambda }^{r\delta }=\sum\nolimits_{\mu } \alpha ^{|\mu |-|\lambda|} \, a_{\lambda ,\mu }(\alpha ) \, m_{\mu }.  \label{=Jm}
\end{equation}
\end{defn}

\begin{conj}[{\protect\cite[ Conjecture 7]{KnS1}}]
The coefficients $a_{\lambda ,\mu }(\alpha )$ belong to $\mathbb{N}%
_{0}[\alpha ].$
\end{conj}

We prove this conjecture in Theorem \ref{thm:A} below.

\subsection{Nonsymmetric polynomials}\label{sec:nonsym}

Nonsymmetric interpolation polynomials are indexed by compositions $\eta \in
\mathbb{N}^{n}$ and their coefficients depend on $\rho =\left( \rho
_{1},\ldots ,\rho _{n}\right) $ as before. For $\gamma \in \mathbb{N}^{n}$ let $w_{\gamma }$ be the shortest permutation such that $\gamma
^{+}=w_{\gamma }^{-1}(\gamma )$ is a partition and define%
\begin{equation}
\overline{\gamma }=\gamma +w_{\gamma }\left( \rho \right) =w_{\gamma }\left(
\gamma ^{+}+\rho \right) .  \label{=gambar}
\end{equation}%
We note that for a partition $\mu $ we have $\mu =\mu ^{+}$ and $w_{\mu }=1$
and hence $\overline{\mu }=\mu +\rho $.

\begin{thm}[\protect\cite{Kn1, Sa2}]
There is a unique polynomial $E_{\eta }^{\rho }\left( x\right) =E_{\eta
}^{\rho }(x_{1},\ldots ,x_{n})$ of total degree $\left\vert \eta \right\vert
={\eta _{1}}+\cdots +{\eta _{n}}$ \ such that

\begin{enumerate}
\item $E_{\eta }^{\rho }\left( \overline{\gamma }\right) =0$ \ for all $%
\gamma \in \mathbb{N}^{n}$ such that $\left\vert \gamma \right\vert
\leq \left\vert \eta \right\vert ,\gamma \neq \eta $.

\item The coefficient of the monomial $x^{\eta }$ in $E_{\eta }^{\rho }$ is $%
1$.
\end{enumerate}
\end{thm}

As before, this is equivalent to the following interpolation result.

\begin{thm}[\protect\cite{Kn1, Sa2}]
\label{N-inter}A polynomial of degree $d$ is uniquely characterized by its
values on the set $\left\{ \overline{\gamma }:\left\vert \gamma \right\vert
\leq d\right\} .$
\end{thm}

This is proved in \cite{Kn1, Sa2} for various special choices of $\rho $, but the argument works in general. Indeed the interpolation conditions mean that the coefficients of the polynomial satisfy a (square) system of linear equations over the field $\mathbb{Q}\left( \rho _{1},\ldots ,\rho_{n}\right) $. What we need to show is that the determinant of the corresponding matrix is not identically zero. Thus the result for any special $\rho $ actually \emph{implies} the result for generic $\rho $.

For the special choice $\rho =r\delta $ the interpolation polynomials are related to nonsymmetric Jack polynomials \cite{KnS2, O}.

\begin{thm}[\protect\cite{Kn1}]
\label{Etop}For $\rho =r\delta $ we have%
\begin{equation*}
E_{\eta }^{r\delta }=E_{\eta }^{(\alpha )}+{\text{ terms of degree}<|\eta |},
\end{equation*}%
where $E_{\eta }^{(\alpha )}$ is the nonsymmetric Jack polynomial with
parameter $\alpha =1/r$.
\end{thm}

This is proved in \cite{Kn1} for a slightly different polynomial, denoted $%
E_{\eta }$ in \cite{Kn1} and $G_{\eta }$ in \cite{Sa3}, which is defined
with respect to 
\begin{equation}
\rho =\left( 0,-r,\ldots ,-\left( n-1\right) r\right) =r\delta -\left(
nr-r\right) \mathbf{1}  \label{=arho}
\end{equation}%
where $\mathbf{1}=\left( 1,\ldots ,1\right) $. It follows easily that 
\begin{equation}
E_{\eta }^{r\delta }\left( x\right) =G_{\eta }\left( x+\left( nr-r\right) 
\mathbf{1}\right) .  \label{=EG}
\end{equation}%
In particular $E_{\eta }^{r\delta }$ has the same top degree part as $%
G_{\eta }$, namely $E_{\eta }^{(\alpha )}$.

In \cite[Sec. 4]{KnS2}, Knop and Sahi defined the normalized
nonsymmetric Jack polynomials 
\begin{equation}
F_{\eta }^{(\alpha )}=\,d_{\eta }(\alpha )\,E_{\eta }^{(\alpha )},
\label{=Fa}
\end{equation}%
where the normalizing factor $d_{\eta }(\alpha )$ is a product over boxes in
the Ferrers diagram of $\eta $, \textit{i.e.} over pairs $s=\left(
i,j\right) $ such that $j\leq \eta _{i}$. Explicitly we have 
\begin{equation*}
d_{\eta }(\alpha )=\prod\nolimits_{s\in \lambda }\left( \alpha \left(
a_{\eta }\left( s\right) +1\right) +l_{\eta }\left( s\right) +1\right) ,
\end{equation*}%
where $a_{\eta }$ and $l_{\eta }$ are the \emph{arm} and \emph{leg} of $%
s=\left( i,j\right) $ defined by 
\begin{equation}
a_{\eta }\left( i,j\right) =\eta _{i}-j,\quad l_{\eta }\left( i,j\right)
=\#\left\{ k>i:j\leq \eta _{k}\leq \eta _{i}\right\} +\#\left\{ k<i:j\leq
\eta _{k}+1\leq \eta _{i}\right\} .  \label{=l-eta}
\end{equation}

The main result of \cite[Sec. 4]{KnS2} is as follows.

\begin{thm}[\protect\cite{KnS2}]
\label{thm:F-integ}The coefficients of $F_{\eta }^{(\alpha )}$ with respect
to the monomials $x^{\gamma }$ belong to $\mathbb{N}[\alpha ]$.
\end{thm}

Our Theorem \ref{thm:B} is a generalization of this result for interpolation
polynomials. In analogy with Definition \ref{defn: symm-expansion} we make
the following definition.

\begin{defn}
The normalized nonsymmetric interpolation polynomial is 
\begin{equation}
F_{\eta }^{r\delta }(x)=\left( -1\right) ^{{|\eta |}} d_{\eta }( \alpha ) \,
E_{\eta }^{r\delta }\left( -x\right)  \label{=Frd}
\end{equation}%
and its expansion coefficients $b_{\eta ,\gamma }(\alpha )$ are defined by%
\begin{equation}
F_{\eta }^{r\delta }=\sum\nolimits_{\gamma }\alpha ^{|\gamma |-|\eta |} \,
b_{\eta ,\gamma }(\alpha ) \, x^{\gamma }.  \label{=Fx}
\end{equation}
\end{defn}

In Theorem \ref{thm:B} we show that the $b_{\eta ,\gamma }(\alpha )$ belong
to $\mathbb{N}[\alpha ]$.

\subsection{Intertwiners and recursion\label{subsec:rec}}

Symmetric polynomials arise naturally as special functions in representation
theory and combinatorics. However, in the context of the present paper,
nonsymmetric polynomials are easier to work with because they satisfy useful
recursions with respect to the symmetric group. The simplest manifestation
of this phenomenon involves ordinary monomials, which can be generated from $%
x^{0}=1$ by the recursions%
\begin{equation*}
x^{s_{i}\eta }=s_{i}\left( x^{\eta }\right) ,\quad x^{\Phi \eta }=\Phi
\left( x^{\eta }\right) .
\end{equation*}%
Here $s_{i}$ is the elementary transposition which interchanges $\eta _{i}$
and $\eta _{i+1}$, and which acts on functions by interchanging $x_{i}$ and $%
x_{i+1}$, while $\Phi $ is the \textquotedblleft affine intertwiner" which
acts by 
\begin{equation}
\Phi \eta =\left( \eta _{2},\ldots ,\eta _{n},\eta _{1}+1\right) \text{%
,\quad }\Phi f\left( x\right) =x_{n}f(x_{n},x_{1},\ldots ,x_{n-1}).
\label{=Phi}
\end{equation}%
Thus $\Phi $ is the translation $\eta \mapsto \left( \eta _{1}+1,\eta
_{2},\ldots ,\eta _{n}\right) $ followed by the $n$-cycle 
\begin{equation}
\omega =s_{1}\cdots s_{n-1}=\left( 1,2,\ldots ,n\right)  \label{=omega}
\end{equation}

The corresponding result for Jack polynomials involves the scalars 
\begin{equation}
c_{i}^{\eta }=\frac{r}{\overline{\eta }_{i}-\overline{\eta }_{i+1}}\quad 
\text{and}\quad d_{i}^{\eta }=\left\{ 
\begin{tabular}{cl}
$1$ & if $\eta _{i}<\eta _{i+1}$ \\ 
$1-\left( c_{i}^{\eta }\right) ^{2}$ & if\ $\eta _{i}\geq \eta _{i+1}$.%
\end{tabular}%
\right.  \label{=cd}
\end{equation}

\begin{thm}[\protect\cite{KnS1}]
\label{Hrec}Nonsymmetric Jack polynomials satisfy the recursions\ 
\begin{equation}
E_{\Phi \eta }^{\left( \alpha \right) }=\Phi \left( E_{\eta }^{\left( \alpha
\right) }\right) ,\quad \left( s_{i}+c_{i}^{\eta }\right) E_{\eta }^{\left(
\alpha \right) }=d_{i}^{\eta }E_{s_{i}\eta }^{\left( \alpha \right) }.
\label{=Ea-rec}
\end{equation}
\end{thm}

The $\Phi $-relation is \cite[Cor 4.2]{KnS1}. The $s_{i}$-relation is proved
for $\eta _{i}<\eta _{i+1}$ in \cite[Prop 4.3]{KnS1} and for $\eta _{i}=\eta
_{i+1}$ in \cite[Lemma 2.4]{KnS1}. In the latter situation we have $%
c_{i}^{\eta }=-1$ and $d_{i}^{\eta }=0$ so that the $s_{i}$-relation reduces
to $s_{i}E_{\eta }^{\left( \alpha \right) }=E_{\eta }^{\left( \alpha \right)
}$ as in \cite[Lemma 2.4]{KnS1}. The remaining case $\eta _{i}>\eta _{i+1}$
follows readily by applying $s_{i}$ to both sides of the relation for the
case $\eta _{i}<\eta _{i+1}$.

The analogous result for interpolation polynomials involves the operators%
\begin{equation}
\partial _{i}\left( f\right) =\frac{s_{i}\left( f\right) -f}{x_{i}-x_{i+1}}%
\text{,\quad }\sigma _{i}^{-}=s_{i}-r\partial _{i},\quad \Phi ^{-}f\left(
x\right) =x_{n}f(x_{n}-1,x_{1},\ldots ,x_{n-1}).  \label{=sFm}
\end{equation}

\begin{thm}[\protect\cite{Kn1}]
\label{Irec}Nonsymmetric interpolation polynomials satisfy the recursions 
\begin{equation*}
E_{\Phi \eta }^{r\delta }=\Phi ^{-}E_{\eta }^{r\delta }\text{,\quad }\left(
\sigma _{i}^{-}+c_{i}^{\eta }\right) E_{\eta }^{r\delta }=d_{i}^{\eta
}E_{s_{i}\eta }^{r\delta }.
\end{equation*}
\end{thm}

This is proved in \cite{Kn1} for the variant $G_{\eta }$ corresponding to $%
\rho $ as in (\ref{=arho}), and by (\ref{=EG}) it implies the result for $%
E_{\eta }^{r\delta }$.

\begin{rem} \label{rem:Erecur}
These recursions suffice to generate all $E_{\eta }^{r\delta }$: Suppose $\eta \neq 0$. Let $i$ be the largest index such that $\eta _{i}\neq 0$. If $i=n$, then 
$$
E_{\eta }^{r\delta }=\Phi^{-}\left( E_{\gamma }^{r\delta }\right)
$$
where $\gamma = (\eta_n-1,\eta_1,\eta_2,\dots,\eta_{n-1})$. Otherwise, 
$$
E_{\eta }^{r\delta }=\frac{1}{d_{i}^{s_i(\eta) }}\left( \sigma _{i}^{-}+c_{i}^{s_i(\eta) }\right)
E_{s_i(\eta) }^{r\delta }
$$
Applying these identities repeatedly, we eventually reach the case $E_{0}^{r\delta}=1$. We can generate all $E_{\eta }^{\left( \alpha \right) }$ in a similar way.
\end{rem}

\section{Bar monomials}

\subsection{The dehomogenization operator}\label{dehom}

The homogeneous polynomials $F_{\eta }^{\left( \alpha \right) }$ and the
inhomogeneous $F_{\eta }^{r\delta }$ are both linear bases for the
polynomial algebra $\mathbb{F}\left[ x_{1},\ldots ,x_{n}\right] $ over the
field $\mathbb{F=Q}\left( r\right) =\mathbb{Q}\left( \alpha \right) $. Thus
there is a unique linear operator on $\mathbb{F}\left[ x_{1},\ldots ,x_{n}%
\right] $ which maps the first basis to the second.

\begin{defn}
The \textit{dehomogenization operator} $\Xi $ is the unique $\mathbb{F}$%
-linear operator satisfying 
\begin{equation}
\Xi (F_{\eta }^{(\alpha )})=F_{\eta }^{r\delta }.  \label{eq:dehomF}
\end{equation}
\end{defn}

We now prove some basic properties of $\Xi $. It is simpler to first
consider the modification $\Psi =S^{-1}\Xi S=S\Xi S$ where $S=S^{-1}$ is the
sign change operator 
\begin{equation*}
Sf\left( x\right) =f\left( -x\right) .
\end{equation*}

\begin{prop}
\label{prop:PsiIntertwine} The operator $\Psi $ maps $E_{\eta }^{\left(
\alpha \right) }$ to $E_{\eta }^{r\delta }$ and satisfies the intertwining
properties%
\begin{equation}
\Phi ^{-}\Psi =\Psi \Phi ,\quad \sigma _{i}^{-}\Psi =\Psi s_{i}\text{.}
\label{=inter}
\end{equation}
\end{prop}

\begin{proof}
Since $E_{\eta }^{\left( \alpha \right) }$ is homogeneous of degree $%
\left\vert \eta \right\vert $ and $\Xi $ is linear we get%
\begin{equation*}
\Xi \left( SE_{\eta }^{\left( \alpha \right) }\right) =\Xi \left( \left(
-1\right) ^{\left\vert \eta \right\vert }E_{\eta }^{\left( \alpha \right)
}\right) =\frac{\left( -1\right) ^{\left\vert \eta \right\vert }}{d_{\eta
}\left( \alpha \right) }\Xi (F_{\eta }^{(\alpha )})=\frac{\left( -1\right)
^{\left\vert \eta \right\vert }}{d_{\eta }\left( \alpha \right) }F_{\eta
}^{r\delta }=S\left( E_{\eta }^{r\delta }\right) ,
\end{equation*}%
whence $\Psi =S^{-1}\Xi S$ maps $E_{\eta }^{\left( \alpha \right) }$ to $%
E_{\eta }^{r\delta }$. \ Next, by Theorems \ref{Hrec} and \ref{Irec} we have 
\begin{eqnarray*}
\Phi ^{-}\Psi \left( E_{\eta }^{\left( \alpha \right) }\right) &=&E_{\Phi
\left( \eta \right) }=\Psi \Phi \left( E_{\eta }^{\left( \alpha \right)
}\right) \\
\left( \sigma _{i}^{-}+c_{i}^{\eta }\right) \Psi \left( E_{\eta }^{\left(
\alpha \right) }\right) &=&d_{i}^{\eta }E_{s_{i}\eta }^{r\delta }=\Psi
\left( s_{i}+c_{i}^{\eta }\right) \left( E_{\eta }^{\left( \alpha \right)
}\right) .
\end{eqnarray*}%
This shows that identities in (\ref{=inter}) hold on the basis $E_{\eta
}^{\left( \alpha \right) }$, and therefore hold in general.
\end{proof}

\begin{prop}
\label{prop:PsiVanish} If $f$ is homogenous then $g=\Psi \left( f\right) $
is characterized by the properties

\begin{enumerate}
\item $g\left( x\right) =f\left( x\right) +{\text{ terms of degree}<\deg }%
\left( f\right) .$

\item $g\left( \overline{\eta }\right) =0$ for all compositions $\eta $ with 
$\left\vert \eta \right\vert <\deg \left( f\right) .$
\end{enumerate}
\end{prop}

\begin{proof}
For $f$ of a fixed homogeneity degree the two properties are linear in $f$.
Therefore it is sufficient to verify them for $f$ $=E_{\eta }^{\left( \alpha
\right) }$. By Proposition \ref{prop:PsiIntertwine} we have $g=E_{\eta
}^{r\delta }$, and by Theorem \ref{Etop} $E_{\eta }^{r\delta }$ satisfies
the two properties.

Now suppose $g_{1}$ and $g_{2}$ both satisfy the two properties. Then the
difference $g_{1}-g_{2}$ has degree ${<\deg }\left( f\right) $ and vanishes
at all $\overline{\eta }$ with $\left\vert \eta \right\vert <\deg \left(
f\right) .$ Thus by Theorem \ref{N-inter} we have $g_{1}-g_{2}=0.$ This
proves the uniqueness of $g$.
\end{proof}

\begin{prop}
\label{prop:Psi-symm} The operator $\Psi $ preserves the space of symmetric
polynomials.
\end{prop}

\begin{proof}
A function $f$ is symmetric iff $s_{i}\left( f\right) =f$ for all $i$. By
the definition of $\sigma _{i}^{-}$ we have%
\begin{equation*}
\sigma _{i}^{-}\left( f\right) -f=\left( 1-\frac{r}{x_{i}-x_{i+1}}\right)
\left( s_{i}\left( f\right) -f\right)
\end{equation*}%
Thus $s_{i}\left( f\right) =f$ if and only if $\sigma _{i}^{-}\left(
f\right) =f$. Now the relation $\sigma _{i}^{-}\Psi =\Psi s_{i}$ (\ref%
{=inter}) shows that if $f$ is symmetric then so is $\Psi \left( f\right) .$
\end{proof}

\begin{prop}
\label{prop:Psi-symm-vanish} If $f$ is homogeneous symmetric then $g=\Psi
\left( f\right) $ is characterized by the properties

\begin{enumerate}
\item $g$ is symmetric

\item $g\left( x\right) =f\left( x\right) +{\text{ terms of degree}<\deg }%
\left( f\right) .$

\item $g\left( \mu +r\delta \right) =0$ for all partitions $\mu $ with $%
\left\vert \mu \right\vert <\deg \left( f\right) .$
\end{enumerate}
\end{prop}

\begin{proof}
By Propositions \ref{prop:Psi-symm} and \ref{prop:PsiVanish} $g=\Psi \left(
f\right) $ satisfies the three properties, and the uniqueness follows from
Theorem \ref{S-inter}.
\end{proof}

\begin{prop}
\label{prop:Psi-P}The operator $\Psi $ maps $P_{\lambda }^{\left( \alpha
\right) }$ to $P_{\lambda }^{r\delta }$.
\end{prop}

\begin{proof}
This is immediate from Proposition \ref{prop:Psi-symm-vanish} and Theorems %
\ref{thm:Pdef}, \ref{S-inter}.
\end{proof}

Proposition \ref{prop:Psi-P} shows that the restriction of $\Psi $ to
symmetric polynomials is the operator studied in \cite[Sec. 6]{KnS1} in
connection with the Pieri formula for interpolation polynomials.

We now set $\sigma _{i}^{+}=S^{-1}\sigma _{i}S$ and $\Phi ^{+}=S^{-1}\Phi
^{-}S$ so that we have 
\begin{equation}
\sigma _{i}^{+}=s_{i}+r\partial _{i},\quad \Phi ^{+}f\left( x\right)
=x_{n}f(x_{n}+1,x_{1},\ldots ,x_{n-1}).  \label{=Fsp}
\end{equation}

\begin{thm}
\label{thm:Xinter}The operator $\Xi $ satisfies the intertwining properties%
\begin{equation}
\Phi ^{+}\Xi =\Xi \Phi ,\quad \sigma _{i}^{+}\Xi =\Xi s_{i}\text{.}
\label{=Xinter}
\end{equation}
\end{thm}

\begin{proof}
This is immediate from Proposition \ref{prop:PsiIntertwine}
\end{proof}

\begin{thm}
\label{thm:XiVanish}If $f$ is homogeneous then $g=\Xi \left( f\right) $ is
characterized by the properties

\begin{enumerate}
\item $g\left( x\right) =f\left( x\right) +{\text{ terms of degree}<\deg }%
\left( f\right) .$

\item $g\left( -\overline{\eta }\right) =0$ for all compositions $\eta $
with $\left\vert \eta \right\vert <\deg \left( f\right) .$
\end{enumerate}
\end{thm}

\begin{proof}
This is immediate from Proposition \ref{prop:PsiVanish}.
\end{proof}

\begin{thm}
\label{thm:symmXiVanish}The operator $\Xi $ preserves the space of symmetric
polynomials and maps $J_{\lambda }^{\left( \alpha \right) }$ to $J_{\lambda
}^{r\delta }$. If $f$ is homogeneous symmetric then $g=\Xi \left( f\right) $
is characterized by the properties

\begin{enumerate}
\item $g$ is symmetric.

\item $g\left( x\right) =f\left( x\right) +{\text{ terms of degree}<\deg }%
\left( f\right) .$

\item $g\left( -\mu -r\delta \right) =0$ for all partitions $\mu $ with $%
\left\vert \mu \right\vert <\deg \left( f\right) .$
\end{enumerate}
\end{thm}

\begin{proof}
This is immediate from Propositions \ref{prop:Psi-symm}, \ref%
{prop:Psi-symm-vanish} and \ref{prop:Psi-P}.
\end{proof}

\subsection{The bar monomials}\label{bar-mon}

We now consider the action of the dehomogenization operator on the monomial 
\begin{equation*}
x^{\eta }=x_{1}^{\eta _{1}}x_{2}^{\eta _{2}}\cdots x_{n}^{\eta _{n}}.
\end{equation*}

\begin{defn}
\label{defn:barmon}The \emph{bar monomial} corresponding to a composition $%
\eta $ is 
\begin{equation*}
x^{\underline{\eta }}=\Xi \left( x^{\eta }\right)
\end{equation*}
\end{defn}

We note that the bar monomial is \emph{not} a monomial, however by Theorem %
\ref{thm:XiVanish} it is a monomial up to lower degree terms.

\begin{thm}
\label{thm:bar-inter}The bar monomial $x^{\underline{\eta }}$ is the unique
polynomial $g\left( x\right) $ satisfying

\begin{enumerate}
\item $g\left( x\right) =x^{\eta }+$ terms of degree $<\left\vert \eta
\right\vert $

\item $g\left( -\overline{\gamma }\right) =0$ if $\left\vert \gamma
\right\vert <\left\vert \eta \right\vert $
\end{enumerate}
\end{thm}

\begin{proof}
This immediate from Theorem \ref{thm:XiVanish}.
\end{proof}

\begin{expl}
\label{ex:all2bars} The three bar monomials for $n=2$ and $\left\vert \eta
\right\vert =2$ are as follows:%
\begin{align*}
x^{\underline{(2,0)}}& =(x_{1}+1+r)(x_{1}+r)+r(x_{2}) \\
x^{\underline{(1,1)}}& =(x_{1})(x_{2}) \\
x^{\underline{(0,2)}}& =(x_{2}+1+r)(x_{2})
\end{align*}%
They satisfy the properties of Theorem \ref{thm:bar-inter}. They have the
appropriate top degree term, and each vanishes at $-\overline{\gamma }$ with 
$\left\vert \gamma \right\vert <2$, \ i.e. at the points 
\begin{equation*}
-\overline{(0,0)}=({-}r,0),\quad -\overline{(1,0)}=({-}1{-}r,0),\quad -%
\overline{(0,1)}=(0,{-}1{-}r).
\end{equation*}
\end{expl}

We now establish the basic recursive properties of the bar monomials.

\begin{thm}
\label{thm:bar-recur}The bar monomials satisfy the recursions%
\begin{equation*}
x^{\underline{s_{i}\eta }}=\sigma _{i}^{+}\left( x^{\underline{\eta }%
}\right) ,\quad x^{\underline{\Phi \eta }}=\Phi ^{+}\left( x^{\underline{%
\eta }}\right) .
\end{equation*}
\end{thm}

\begin{proof}
By Theorem \ref{thm:Xinter} we have 
\begin{equation*}
x^{\underline{s_{i}\eta }}=\Xi \left( x^{s_{i}\eta }\right) =\Xi \left(
s_{i}x^{\eta }\right) =\sigma _{i}^{+}\Xi \left( x^{\eta }\right) =\sigma
_{i}^{+}\left( x^{\underline{\eta }}\right) .
\end{equation*}%
The argument for $x^{\underline{\Phi \eta }}$ is entirely analogous.
\end{proof}

\begin{rem}
Just as in Remark \ref{rem:Erecur}, it is easy to see that these recursions generate all bar monomials. We make this explicit in the proof of Theorem \ref{thm:E}, where it plays a central role.
\end{rem}

We now formulate the symmetric analogues of the above ideas.

\begin{defn}
\label{defn:symmbarmon}The \emph{symmetric bar monomial} corresponding to a
partition $\lambda $ is 
\begin{equation*}
m_{\underline{\lambda }}=\Xi \left( m_{\lambda }\right) .
\end{equation*}
\end{defn}

\begin{thm}
\label{thm:symmbar-inter}$m_{\underline{\lambda }}$ is the unique polynomial 
$g\left( x\right) $ satisfying

\begin{enumerate}
\item $g\left( x\right) $ is symmetric

\item $g\left( x\right) =m_{\underline{\lambda }}+$ terms of degree $%
<\left\vert \lambda \right\vert $

\item $g\left( -\overline{\mu }\right) =0$ if $\left\vert \mu \right\vert
<\left\vert \lambda \right\vert $
\end{enumerate}
\end{thm}

\begin{proof}
This immediate from Theorem \ref{thm:symmXiVanish}.
\end{proof}

For any two compositions $\eta, \gamma$, we write $\eta \sim \gamma$ if
one is a rearrangement of the other. 

\begin{prop}
\label{prop:mx}We have $m_{\underline{\lambda }}=\sum\nolimits_{\eta \sim
\lambda }x^{\underline{\eta }}$.
\end{prop}

\begin{proof}
This follows from the homogeneous version $m_{\lambda }=\sum\nolimits_{\eta
\sim \lambda }x^{\eta }$ by applying $\Xi $.
\end{proof}

\begin{expl}
The two symmetric bar monomials for $n=2$ and $\left\vert \lambda
\right\vert =2$ are as follows:%
\begin{align*}
m_{\underline{\left( 1,1\right) }}& =x^{\underline{(1,1)}}=x_{1}x_{2} \\
m_{\left( 2,0\right) }& =x^{\underline{(2,0)}}+x^{\underline{(0,2)}%
}=(x_{1}+1+r)(x_{1}+r)+r(x_{2})+(x_{2}+1+r)(x_{2}) \\
& =x_{1}^{2}+x_{2}^{2}+\left( 1+2r \right) \left( x_{1}+x_{2}\right) +r\left(
1+r\right)
\end{align*}%
They satisfy the properties of Theorem \ref{thm:symmbar-inter}. That is, each is a
symmetric polynomial with the appropriate top degree terms, and vanishes at $%
-\overline{\mu }$ with $\left\vert \mu \right\vert <2$, \ i. e. at the
points 
\begin{equation*}
-\overline{(0,0)}=({-}r,0),\quad -\overline{(1,0)}=({-}1{-}r,0).
\end{equation*}
\end{expl}

\subsection{Proofs of Theorems \protect\ref{thm:A} and \protect\ref{thm:B}}

\label{subsec:PfAB}

The bar monomials in the examples above are polynomials in $x_{1},x_{2}$ and 
$r$ with positive integral coefficients. We will show that this true in
general.

\begin{defn}
The expansion coefficients of the bar monomials are defined by%
\begin{equation*}
x^{\underline{\eta }}=\sum\nolimits_{\gamma }c_{\eta ,\gamma }(r) \,
x^{\gamma },\quad m_{\underline{\lambda }}=\sum\nolimits_{\mu }d_{\lambda
,\mu}(r) \, m_{\mu }.
\end{equation*}
\end{defn}

\setcounter{letterthm}{2}

\begin{letterthm}
\label{thm:C1} The coefficient $c_{\eta ,\gamma }(r)$ is a polynomial in $%
\mathbb{N}[r]$ of degree $\leq \left\vert \eta \right\vert -\left\vert
\gamma \right\vert $.
\end{letterthm}

We prove this in Subsection \ref{PfC} below, but we first deduce
some important consequences. In view of Proposition \ref{prop:mx} we have an
analogous result for $d_{\lambda ,\mu }(r)$.

\begin{cor}
\label{cor:cd}The coefficient $d_{\lambda ,\mu }(r)$ is a polynomial in $%
\mathbb{N}[r]$ of degree $\leq \left\vert \lambda \right\vert
-\left\vert \mu \right\vert $.
\end{cor}

\begin{proof}
By Proposition \ref{prop:mx} we have 
\begin{equation*}
m_{\underline{\lambda }}=\sum\nolimits_{\eta \sim \lambda }x^{\underline{\eta%
}} =\sum\nolimits_{\eta \sim \lambda }\sum\nolimits_{\gamma }c_{\eta ,\gamma
}(r) \, x^{\gamma }=\sum\nolimits_{\gamma }\left[ \sum\nolimits_{\eta \sim
\lambda }c_{\eta ,\gamma }(r)\right] x^{\gamma }.
\end{equation*}

Comparing the coefficients of $x^{\mu}$ on both side, we get 
\begin{equation*}
d_{\lambda ,\mu }(r)=\sum\nolimits_{\eta \sim \lambda }c_{\eta ,\mu }(r).
\end{equation*}%
Now the result follows from Theorem \ref{thm:C}.
\end{proof}

We can also prove Theorems \ref{thm:A} and \ref{thm:B}.

\begin{proof}[Proof of Theorem \protect\ref{thm:B}]
The nonsymmetric interpolation polynomials and Jack polynomials have
expansions 
\begin{equation}
F_{\eta }^{r\delta }=\sum\nolimits_{\left\vert \gamma \right\vert \leq
\left\vert \eta \right\vert }\alpha ^{|\gamma |-|\eta |}\,b_{\eta ,\gamma
}(\alpha )\,x^{\gamma },\quad \ F_{\eta }^{\left( \alpha \right)
}=\sum\nolimits_{\left\vert \zeta \right\vert =\left\vert \eta \right\vert
}b_{\eta ,\zeta }(\alpha )\,x^{\zeta },  \label{=bhg}
\end{equation}%
and by Theorem \ref{thm:F-integ} we have 
\begin{equation}
b_{\eta ,\zeta }(\alpha )\in \mathbb{N}[\alpha ]\ \text{ for }\ |\zeta
|=|\eta |.  \label{=bhk}
\end{equation}

Since $F_{\eta }^{r\delta }=\Xi \left( F_{\eta }^{\left( \alpha \right)
}\right) $ we get%
\begin{equation*}
F_{\eta }^{r\delta }=\sum\nolimits_{\left\vert \zeta \right\vert =\left\vert
\eta \right\vert }b_{\eta ,\zeta }(\alpha ) \, x^{\underline{\zeta }%
}=\sum\nolimits_{\left\vert \zeta \right\vert =\left\vert \eta \right\vert
}b_{\eta ,\zeta }(\alpha )\sum\nolimits_{\gamma }c_{\zeta ,\gamma
}(r)\,x^{\gamma }
\end{equation*}%
which implies that 
\begin{equation*}
b_{\eta ,\gamma }(\alpha )=\sum\nolimits_{\left\vert \zeta \right\vert
=\left\vert \eta \right\vert }b_{\eta ,\zeta }(\alpha ) \, \tilde{c}_{\zeta
,\gamma }(\alpha )
\end{equation*}%
where 
\begin{equation*}
\tilde{c}_{\zeta ,\gamma }(\alpha )=\alpha ^{|\eta |-\left\vert \gamma
\right\vert } \, c_{\zeta ,\gamma }(r)=\alpha ^{|\zeta |-\left\vert \gamma
\right\vert } \, c_{\zeta ,\gamma }(r).
\end{equation*}

Rewriting Theorem \ref{thm:C} in terms of $\alpha =1/r$ we have 
\begin{equation}
\alpha ^{|\zeta |-\left\vert \gamma \right\vert }\,c_{\zeta ,\gamma }(r)\in 
\mathbb{N}[\alpha ].  \label{=c-alpha}
\end{equation}%
Together with (\ref{=bhk}) this implies that $b_{\eta ,\gamma }(\alpha )\in 
\mathbb{N}\left[ \alpha \right] $, proving Theorem \ref{thm:B}.
\end{proof}

\begin{proof}[Proof of Theorem \protect\ref{thm:A}]
In the symmetric case we get the formula%
\begin{eqnarray*}
a_{\lambda ,\mu }(\alpha ) &=&\sum\nolimits_{\left\vert \nu \right\vert
=\left\vert \lambda \right\vert }a_{\lambda ,\nu }(\alpha )\,\tilde{d}_{\nu
,\mu }(\alpha ) \\
\tilde{d}_{\nu ,\mu }(\alpha ) &=&\sum\nolimits_{\eta \sim \nu }\tilde{c}%
_{\eta ,\mu }(\alpha ).
\end{eqnarray*}%
Arguing as above we get $a_{\lambda ,\mu }(\alpha )\in \mathbb{N}\left[
\alpha \right] $, proving Theorem \ref{thm:A}.
\end{proof}

\section{Bar games}\label{bar-games}

In this section we introduce some new combinatorial objects related to
compositions. These objects will be the summation indices in Theorem \ref{thm:D},
the combinatorial expression for the bar monomials. We will prove this 
using Theorem \ref{thm:bar-recur}. As such, it will be important to understand how the weights of our objects behave under the operators $\sigma_i^+$ and $\Phi^+$, and how  compositions behave under $s_i$ and $\Phi$.

\subsection{The critical box}\label{crit-box}

Our main combinatorial object will be called a \emph{bar game}. A game will consist of moves. Each move will begin by deleting a prescribed box from a composition, which we will call the \emph{critical box}.

\begin{defn}
We define the \emph{critical} \emph{box} of a composition $\eta $ to be $s%
\left[ \eta \right] =\left( k,m\right) $ where%
\begin{equation*}
m=m\left[ \eta \right] :=\max \left\{ \eta _{i}\right\} ,\quad k=k\left[
\eta \right] :=\min \left\{ i:\eta _{i}=m\right\} .
\end{equation*}%
We will call $k=k\left[ \eta \right] $ the \emph{critical row} and \emph{\ }$%
l\left[ \eta \right] :=l_{\eta }\left( k,m\right) $ the \emph{critical leg}.
\end{defn}

Alternatively $k=k\left[ \eta \right] $ is characterized by 
\begin{equation}
\eta _{k}>\eta _{1},\ldots ,\eta _{k-1}\text{ and }\eta _{k}\geq \eta
_{k+1},\ldots \eta _{n}.  \label{=kchk}
\end{equation}%
Then we have $m=m\left[ \eta \right] =\eta _{k},$ and the formula (\ref%
{=l-eta}) for $l_{\eta }\left( k,m\right) $ becomes 
\begin{equation}
l\left[ \eta \right] =\#\left\{ i>k:\eta _{i}=m\right\} +\#\left\{ i<k:\eta
_{i}=m-1\right\}  \label{=leta}
\end{equation}

We now discuss the behavior of these quantities under the maps $s_{i}$, $%
\Phi $, and $\omega $ where%
\begin{equation*}
\Phi \left( \eta \right) =\left( \eta _{2},\ldots ,\eta _{n},\eta
_{1}+1\right) 
\text{\quad and \quad }\omega \left( \eta \right) =\left( \eta
_{2},\ldots ,\eta _{n},\eta _{1}\right) .
\end{equation*}

\begin{prop}
\label{prop:crit}Suppose the critical box of $\eta $ is $s\left[ \eta \right]
=\left( k,m\right) $.

\begin{enumerate}
\item If $k>1$ then $s\left[ \Phi \eta \right] =\left( k-1,m\right) $; if $%
k=1$ then $s\left[ \Phi \eta \right] =\left( n,m+1\right) .$

\item If $s_{i}\eta \neq \eta $ then $s\left[ s_{i}\eta \right] =\left(
s_{i}\left( k\right) ,m\right) $
\end{enumerate}
\end{prop}

\begin{proof}
Since $m\left[ \eta \right] $ is the length of the critical row $k\left[
\eta \right] $, it suffices to prove that critical rows of $\Phi \eta $ and $%
s_{i}\eta $ are $\omega \left( k\right) $ and $s_{i}\left( k\right) $,
respectively. In the case of $\Phi \eta $ this comes down to the following
inequalities which are immediate from (\ref{=kchk})%
\begin{gather*}
\eta _{1}+1>\eta _{2},\ldots ,\eta _{n}\text{ if }k=1, \\
\eta _{k}>\eta _{2},\ldots ,\eta _{k-1}\text{ }\text{and }\eta _{k}\geq \eta
_{k+1},\ldots \eta _{n},\eta _{1}+1\text{ if }k>1,
\end{gather*}%
For the case of $s_{i}\eta $ since $\left( s_{i}\eta \right) _{s_{i}\left(
j\right) }=\eta _{j}$ it suffices to show 
\begin{equation*}
s_{i}\left( j\right) <s_{i}\left( k\right) \implies \eta _{j}<\eta _{k}
\end{equation*}%
Except if $k=i,j=i+1$ the condition $s_{i}\left( j\right) <s_{i}\left(
k\right) $ implies $j<k$ and hence $\eta _{j}<\eta _{k}$ For $k=i,j=i+1$ we
need to show $\eta _{k+1}<\eta _{k}$. Now by definition of $k=k\left[ \eta %
\right] $ we have $\eta _{k+1}\leq \eta _{k}$ and since $k=i$, the
assumption $s_{i}\eta \neq \eta $ implies $\eta _{k+1}\neq \eta _{k}$.
\end{proof}

The critical leg $l\left[ \eta \right] $ behaves as follows.

\begin{lem}
We have $l\left[ \Phi \eta \right] =l\left[ \eta \right] $; moreover $l\left[
s_{i}\eta \right] =l\left[ \eta \right] $ except in the following two cases

\begin{enumerate}
\item $l\left[ s_{i}\eta \right] =l\left[ \eta \right] +1$ \ if $k\left[ \eta %
\right] =i$ and $\eta _{i+1}=\eta _{i}-1$

\item $l\left[ s_{i}\eta \right] =l\left[ \eta \right] -1$ \ if $k\left[ \eta %
\right] =i+1$ and $\eta _{i}=\eta _{i+1}-1$
\end{enumerate}
\end{lem}

\begin{proof}
This is immediate from (\ref{=leta}) and Proposition \ref{prop:crit}.
\end{proof}

\begin{defn}
We write $\eta ^{\ast }$ for the composition obtained from $\eta $ by
deleting the critical box.
\end{defn}

Then Proposition \ref{prop:crit} immediately implies the following result.

\begin{cor}
\label{cor: eta-star} We have $\left[ \Phi \left( \eta \right) \right]
^{\ast }=\Phi \left( \eta ^{\ast }\right) $, and if $s_{i}\left( \eta
\right) \neq \eta $ then $\left( s_{i}\eta \right) ^{\ast }=s_{i}\left( \eta
^{\ast }\right) $.
\end{cor}

\subsection{Glissades and the bar order}\label{gliss}

We consider the following operation on compositions that we call a \emph{%
glissade}.  (These will be the moves of our games, which are introduced in the next subsection.)

\vspace{6pt}
\begin{center}
\parbox{36em}{\tt Delete the critical box to get $\eta ^{\ast }$, and then move $l$ $\geq$ $0$ boxes from the end of the critical row $k$ to the end of some other row $j$, with the proviso that the new positions of the boxes are either
above and strictly left, or below and weakly left of their original
positions. }
\end{center}

\smallskip
\begin{expl}
Some examples of glissades can be found in Figures \ref{fig:introex}, \ref{fig:Precur}, and \ref{fig:ex183025}.
For each glissade, we have placed a $\times$ in the critical box
and indicated movement of other boxes with arrows.
\end{expl}

We write $\eta \gtrdot \gamma $ if $\gamma $ is obtained from $\eta $ by a
glissade. We now discuss how glissades behave under the action of the
operators $s_{i}$ and $\Phi $. In view of Corollary \ref{cor: eta-star} we
focus on the case of glissades $\gamma \neq \eta ^{\ast }$, and thus we
define%
\begin{equation}
P\left[ \eta \right] =\left\{ \gamma : \eta \gtrdot \gamma \right\}
\setminus \left\{ \eta ^{\ast }\right\}  \label{=Peta}
\end{equation}

\begin{prop}
\label{prop:P-rec}We have $P\left[ \Phi \eta \right] =\Phi \left( P\left[
\eta \right] \right) $, and if $s_{i}\eta \neq \eta $ then $P\left[
s_{i}\eta \right] =s_{i}\left( P\left[ \eta \right] \right) $\emph{\ except}
as in the following table%
\begin{equation}
\begin{tabular}{|c|c|c|}
\hline
$i$ & $\eta _{i+1}-\eta _{i}$ & $P\left[ s_{i}\eta \right] $ \\ \hline\hline
$k-1$ & $>1$ & $s_{i}\left( P\left[ \eta \right] \right) \cup \left\{ \eta
^{\ast }\right\} $ \\ \hline
$k$ & $<-1$ & $s_{i}\left( P\left[ \eta \right] \right) \setminus \left\{
\eta ^{\ast }\right\} $ \\ \hline
\end{tabular}
\label{=Ps}
\end{equation}
\end{prop}

\begin{proof}
We denote by $M=\left\langle \gamma ,\eta ,j,k,l\right\rangle $ the 
statement that \textquotedblleft $k=k\left[ \eta \right] $ and $\gamma $ is
obtained from $\eta ^{\ast }$ by moving $l>0$ boxes from row $k$ to row $j$%
\textquotedblright . By Proposition \ref{prop:crit} the statement $M$ is
equivalent to 
\begin{equation*}
\Phi \left( M\right) :=\left\langle \Phi \gamma ,\Phi \eta ,\omega \left(
j\right) ,\omega \left( k\right) ,l\right\rangle \text{ and }s_{i}\left(
M\right) :=\left\langle s_{i}\gamma ,s_{i}\eta ,s_{i}\left( j\right)
,s_{i}\left( k\right) ,l\right\rangle .
\end{equation*}%
Moreover $M=\left\langle \gamma ,\eta ,j,k,l\right\rangle $ represents a
glissade if and only if%
\begin{equation}
\varepsilon =\varepsilon \left( M\right) :=\eta _{k}-1-\gamma _{j}=\eta
_{k}-\eta _{j}-l-1\text{ \quad satisfies\quad\ }\left\{ 
\begin{tabular}{cl}
$\varepsilon >0$ & i$\text{f }j<k$ \\ 
$\varepsilon \geq 0$ & $\text{if }j>k$%
\end{tabular}%
\right.  \label{=gi}
\end{equation}

The $M$-inequality (\ref{=gi}) is \emph{identical} to that for $\Phi \left(
M\right) $ and $s_{i}\left( M\right) $ with the following exceptions where
there is a change in the relative order of $\left( j,k\right) $ and/or a
change in $\varepsilon $: 
\begin{equation*}
\begin{tabular}{|c|c|c|}
\hline
$\left( j,k\right) $ & $M$ & $\Phi \left( M\right) $ \\ \hline\hline
$\left( 1,k\right) $ & $\varepsilon >0$ & $\varepsilon -1\geq 0$ \\ \hline
$\left( j,1\right) $ & $\varepsilon \geq 0$ & $\varepsilon +1>0$ \\ \hline
\end{tabular}%
\quad \quad 
\begin{tabular}{|c|c|c|}
\hline
$\left( j,k\right) $ & $M$ & $s_{i}\left( M\right) $ \\ \hline\hline
$\left( i,i+1\right) $ & $\varepsilon >0$ & $\varepsilon \geq 0$ \\ \hline
$\left( i+1,i\right) $ & $\varepsilon \geq 0$ & $\varepsilon >0$ \\ \hline
\end{tabular}%
.
\end{equation*}%
In each row of the first table the two inequalities are still \emph{%
equivalent}; thus $M$ is a glissade iff $\Phi \left( M\right) $ is a
glissade. The same is true in the second table \emph{except} if $\varepsilon
=0$ which implies that $\gamma =s_{i}\eta ^{\ast }$ and $s_{i}\gamma =\eta
^{\ast }$ and leads to the following two situations:%
\begin{equation*}
\begin{tabular}{|c|c|c|c|}
\hline
$\left( j,k\right) $ & $\eta _{i+1}-\eta _{i}$ & $s_{i}\eta ^{\ast }\in P%
\left[ \eta \right] $ & $\eta ^{\ast }\in P\left[ s_{i}\eta \right] $ \\ 
\hline\hline
$\left( i,i+1\right) $ & \multicolumn{1}{|c|}{$l+1$} & False & True \\ \hline
$\left( i+1,i\right) $ & \multicolumn{1}{|c|}{$-\left( l+1\right) $} & True
& False \\ \hline
\end{tabular}%
.
\end{equation*}%
Since we have $l>0$ we get $l+1>1$ and the above table corresponds precisely
to the exceptions in (\ref{=Ps}). This completes the proof of the
proposition.
\end{proof}

\begin{figure}[h]
\includegraphics[width=3in]{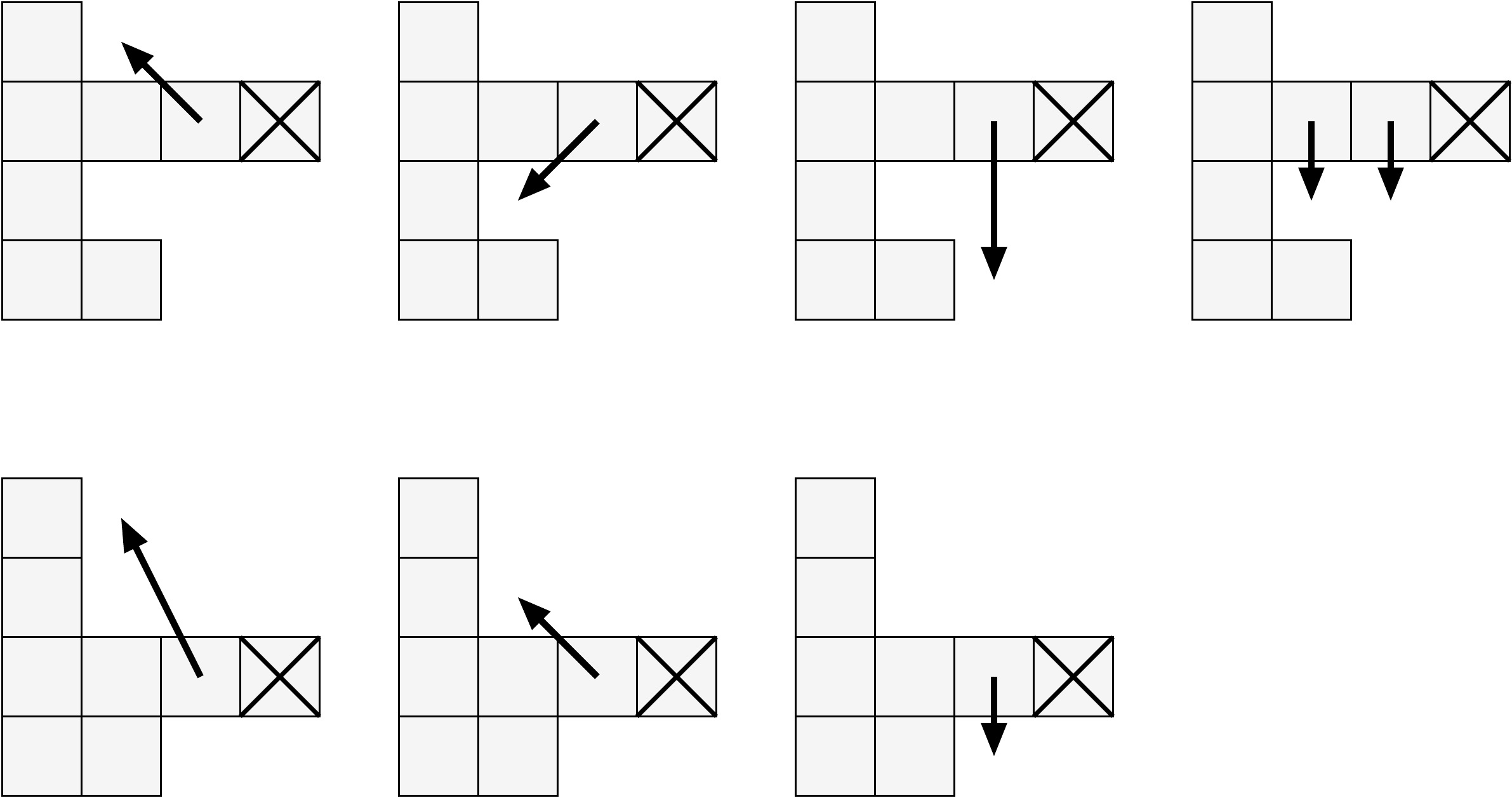}
\caption{All non-trivial glissades on (1,4,1,2) and (1,1,4,2).}
\label{fig:Precur}
\end{figure}

\begin{expl} \label{ex:Precur}
Figure \ref{fig:Precur} shows $P[1,4,1,2]$ and $P[1,1,4,2]$. Notice that there is a glissade on $(1,4,1,2)$ which moves two boxes out of the critical row, but not on $(1,1,4,2)$. This illustrates the special cases in the last table of the Proof of Proposition \ref{prop:P-rec} when $\eta=(1,4,1,2)$ or $\eta=(1,1,4,2)$ and $i=2$.
\end{expl}

\begin{defn}
The \emph{bar order} on compositions is the transitive closure of $\gtrdot $.
\end{defn}

The bar order equips $\left( \mathbb{N}\right) ^{n}$ with the structure
of a ranked poset for which $\gtrdot$ is the covering relation.
The rank function is $\left\vert \eta \right\vert $ and
the composition $0$ is the unique minimal element.

\subsection{Bar games and the proof of Theorem \protect\ref{thm:C}}\label{PfC}

\begin{defn}
A \emph{bar game }on $\eta $ is a maximal $\gtrdot $-chain with greatest
element $\eta $. We write $\mathcal{G}\left( \eta \right) $ for the set of
bar games on $\eta $.
\end{defn}

Each bar game $G$ in $\mathcal{G}\left( \eta \right) $ is a chain of length $%
d=\left\vert \eta \right\vert $ of the form%
\begin{equation}
G:\text{\quad }\eta =\eta ^{\left( 0\right) }\gtrdot \eta ^{\left( 1\right)
}\gtrdot \cdots \gtrdot \eta ^{\left( d\right) }=0.  \label{=eta-seq}
\end{equation}%
We can visualize $\mathcal{G}\left( \eta \right) $ as the set of all
possible \textquotedblleft solitaire\textquotedblright\ games that start
with the Ferrers diagram of $\eta $ and reach $0$ along a sequence of
glissades. There are finitely many games in $\mathcal{G}\left( \eta \right) $%
, each of which ends after exactly $\left\vert \eta \right\vert $ moves.

\begin{figure}[h]
\includegraphics[width=0.9\textwidth]{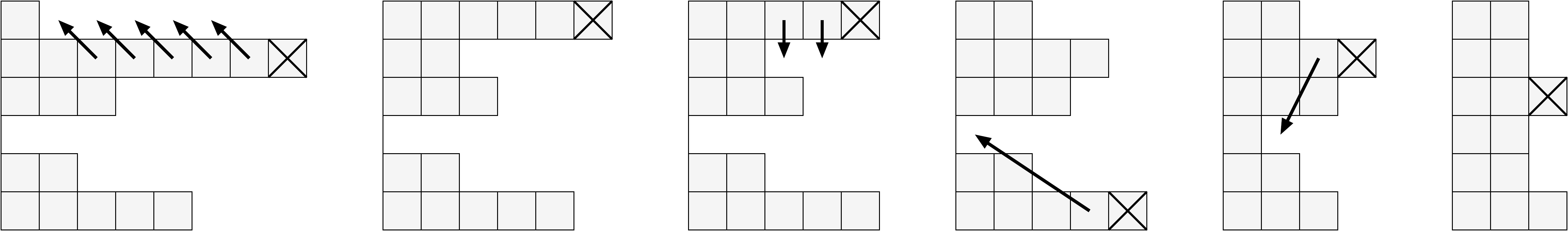}
\caption{A sequence of glissades in a game on (1,8,3,0,2,5)}
\label{fig:ex183025}
\end{figure}

\begin{expl}
\label{ex:183025} Figure \ref{fig:ex183025} shows a bar game on $\eta =
(1,8,3,0,2,5)$. Once we reach the rightmost shape, (2,2,3,2,2,3), 
there is only one possible choice of all future glissades: 
delete the critical box and do nothing else.
The next few shapes will be $(2,2,2,2,2,3)$, $(2,2,2,2,2,2)$, $%
(1,2,2,2,2,2)$, $(1,1,2,2,2,2)$, and so on.
\end{expl}

We now introduce the crucial notion of the weight of a bar game.

\begin{defn}
\label{def:wt}We define the \emph{weight} of a composition $\eta $ with
critical box $\left( k,m\right) $ to be%
\begin{equation*}
w_{\eta }=x_{k}+(m-1)+r\left( n-1-l\left[ \eta \right] \right),
\end{equation*}%
where $l\left[ \eta \right] =l_{\eta }\left( k,m\right) $ is the critical
leg. We define the \emph{weight} of a pair $\eta \gtrdot \gamma$ to be%
\begin{equation*}
w\left( \eta \gtrdot \gamma \right) =\left\{ 
\begin{tabular}{cl}
$w_{\eta }$ & if $\gamma =\eta ^{\ast }$ \\ 
$r$ & if $\gamma \neq \eta ^{\ast }$%
\end{tabular}%
\right. .
\end{equation*}%
We define the \emph{weight }of a game $G$ as in (\ref{=eta-seq}) to be $%
w\left( G\right) =\prod_{i=1}^{d}w\left( \eta ^{\left( i-1 \right) }\gtrdot
\eta ^{\left( i \right) }\right) $.
\end{defn}

\begin{expl}
\label{ex:183025weight} The game in Example \ref{ex:183025} has weight 
\begin{equation*}
r \cdot (x_1+5r+5) \cdot r^3 \cdot (x_3+2r+2) \cdot (x_6+2) \cdot \textstyle{\prod\nolimits_{k=1}^6}(x_k+ 1) \cdot \textstyle{\prod\nolimits_{k=1}^6}x_k .
\end{equation*}
\end{expl}

The connection between bar games and bar monomials is given by Theorerm \ref%
{thm:D} of the introduction, which we now recall in a precise form.

\setcounter{letterthm}{3}

\begin{letterthm}
\label{thm:D1}We have $x^{\underline{\eta }}=\sum_{G\in \mathcal{G}\left(
\eta \right) }w(G)$.
\end{letterthm}

We will prove Theorem \ref{thm:D} in a moment, but we first note that it
immediately implies Theorem \ref{thm:C}.

\begin{proof}[Proof of Theorem \protect\ref{thm:C}]
From Definition \ref{def:wt} each $w(G)$ is a polynomial of total degree $%
\leq \left\vert \eta \right\vert $ in $x_{1},\ldots ,x_{n},r$, with
non-negative integral coefficients; thus the same is true of $x^{\underline{%
\eta }}$.

For the distinguished game $G^{\ast }$ with $\eta ^{\left( i+1\right)
}=\left( \eta ^{\left( i\right) }\right) ^{\ast }$ for all $i$, the monomial 
$x^{\eta }$ occurs once in the expansion of $w\left( G^{\ast }\right) .$ All
other monomials in any $w\left( G\right) $ have degree $<\left\vert \eta
\right\vert $ in $x_{1},\ldots ,x_{n}$. This implies Theorem \ref{thm:C}.
\end{proof}

\subsection{The transition formula and the proof of Theorem \protect\ref%
{thm:D}}\label{PfD}

Bar monomials satisfy the recursions of Theorem \ref{thm:bar-recur} which
involve the operators%
\begin{equation*}
\tilde{\omega}\left( f\right) \left( x\right) =f\left( x_{n}+1,x_{1},\ldots
,x_{n-1}\right) ,\quad \partial _{\iota }f=\frac{s_{i}\left( f\right) -f}{%
x_{i}-x_{i+1}},\quad \Phi ^{+}=x_{n}\tilde{\omega},\quad \sigma
_{i}^{+}=s_{i}+r\partial .
\end{equation*}%
For the proof of Theorem \ref{thm:D} we study their action on the polynomials%
\begin{equation*}
A_{\eta }=\sum\nolimits_{\gamma \in P\left[ \eta \right] }x^{\underline{%
\gamma }},\text{ }B_{\eta }=w_{\eta }x^{\underline{\eta ^{\ast }}},\text{ }%
C_{\eta }=\left( B_{\eta }+rA_{\eta }\right) \text{. }
\end{equation*}

\begin{lem}
\label{lem:Aeta}
We have $\Phi ^{+}\left( A_{\eta }\right) =A_{\Phi \eta }$
and if $\eta _{i}>\eta _{i+1}$ then $\sigma _{i}^{+}\left( A_{\eta }\right)
=A_{s_{i}\eta }$ \textbf{except}

if $i=k$ and $\eta _{i}-1>\eta _{i+1}$then $\sigma _{i}^{+}\left( A_{\eta
}\right) =A_{s_{i}\eta }-r x^{\underline{\eta ^{\ast }}}.$
\end{lem}

\begin{proof}
This is immediate from Theorem \ref{thm:bar-recur} and Proposition \ref%
{prop:P-rec}$.$
\end{proof}

For the action on $B_{\eta }$ we first note the following general result.

\begin{lem}
For any two functions $f,g$ we have%
\begin{equation*}
\Phi \left( fg\right) =\tilde{\omega}\left( f\right) \Phi ^{+}\left(
g\right) ,\quad \sigma _{i}^{+}\left( fg\right) =s_{i}\left( f\right) \sigma
_{i}^{+}\left( g\right) +r\partial _{i}\left( f\right) g
\end{equation*}
\end{lem}

\begin{proof}
The operators $\tilde{\omega}$ and $s_{i}$ are multiplicative 
\begin{equation*}
\tilde{\omega}\left( fg\right) =\tilde{\omega}\left( f\right) \tilde{\omega}%
\left( g\right) ,\quad s_{i}\left( fg\right) =s_{i}\left( f\right)
s_{i}\left( g\right) ,
\end{equation*}%
while $\partial _{\iota }$ is a \textquotedblleft twisted\textquotedblright\
derivation in the following sense%
\begin{equation*}
\partial _{i}\left( fg\right) =\frac{s_{i}\left( f\right) s_{i}\left(
g\right) -s_{i}\left( f\right) g}{x_{i}-x_{i+1}}+\frac{s_{i}\left( f\right)
g-fg}{x_{i}-x_{i+1}}=s_{\iota }\left( f\right) \partial _{i}\left( g\right)
+\partial _{i}\left( f\right) g.
\end{equation*}%
This gives%
\begin{gather*}
\Phi ^{+}\left( fg\right) =x_{n}\tilde{\omega}\left( f\right) \tilde{\omega}%
\left( g\right) =\tilde{\omega}\left( f\right) \Phi ^{+}\left( g\right)  \\
\sigma _{i}^{+}\left( fg\right) =s_{i}\left( f\right) s_{i}\left( g\right) +r
\left[ s_{i}\left( f\right) \partial _{i}\left( g\right) +\partial
_{i}\left( f\right) g\right] =s_{i}\left( f\right) \sigma _{i}^{+}\left(
g\right) +r\partial _{i}\left( f\right) g
\end{gather*}%
as desired.
\end{proof}

We now prove the analog of Lemma \ref{lem:Aeta} for $B_{\eta }$.

\begin{lem}
\label{lem:Beta}We have $\Phi ^{+}\left( B_{\eta }\right) =B_{\Phi \eta }$
and if $\eta _{i}>\eta _{i+1}$ then $\sigma _{i}^{+}\left( B_{\eta }\right)
=B_{s_{i}\eta }$ \textbf{except}

if $i=k$ and $\eta _{i}-1>\eta _{i+1}$then $\sigma _{i}^{+}\left( B_{\eta
}\right) =B_{s_{i}\eta }+rx^{\underline{\eta ^{\ast }}}.$
\end{lem}

\begin{proof}
By Theorem \ref{thm:bar-recur}, Corollary \ref{cor: eta-star} and the
previous lemma we have 
\begin{gather}
\Phi ^{+}\left( B_{\eta }\right) =\tilde{\omega}\left( w_{\eta }\right) \Phi
^{+}\left( x^{\underline{\eta ^{\ast }}}\right) =\tilde{\omega}\left(
w_{\eta }\right) x^{\underline{\Phi \left( \eta ^{\ast }\right) }}=\tilde{%
\omega}\left( w_{\eta }\right) x^{\underline{\left( \Phi \eta \right) ^{\ast
}}}  \label{=PhiB} \\
\sigma _{i}^{+}\left( B_{\eta }\right) =s_{i}\left( w_{\eta }\right) \sigma
_{i}^{+}\left( x^{\underline{\eta ^{\ast }}}\right) +r\partial _{\iota
}\left( w_{\eta }\right) x^{\underline{\eta ^{\ast }}}=s_{i}\left( w_{\eta
}\right) x^{\underline{\left( s_{i}\eta \right) ^{\ast }}}+r\partial _{\iota
}\left( w_{\eta }\right) x^{\underline{\eta ^{\ast }}}  \label{=SiB}
\end{gather}

Now suppose the critical box of $\eta $ is $s\left[ \eta \right] =\left(
k,m\right) $ and the critical leg is $l\left[ \eta \right] =l$ so that 
\begin{equation*}
w_{\eta }=x_{k}+(m-1)+r\left( n-1-l\right).
\end{equation*}

By Proposition \ref{prop:crit} if $k>1$ then $s\left[ \Phi \eta \right]
=\left( k-1,m\right) $ and $l\left[ \Phi \eta \right] =l$ and we get 
\begin{equation*}
w_{\Phi \eta }=x_{k-1}+(m-1)+r\left( n-1-l\right) =\tilde{\omega}\left(
w_{\eta }\right) \text{,}
\end{equation*}%
while if $k=1$ then $s\left[ \Phi \eta \right] =\left( n,m+1\right) $ and $l%
\left[ \Phi \eta \right] =l$ and we get%
\begin{eqnarray*}
w_{\Phi \eta } &=&x_{n}+m+r\left( n-1-l\right) \\
&=&(x_{n}+1) +(m-1)+r\left( n-1-l\right) =\tilde{\omega}\left( w_{\eta }\right) 
\text{,}
\end{eqnarray*}
Thus $\tilde{\omega}\left( w_{\eta }\right) =w_{\Phi \eta }$ always, and by (%
\ref{=PhiB}) we deduce $\Phi ^{+}\left( B_{\eta }\right) =B_{\Phi \eta }.$

By Proposition \ref{prop:crit} if $i\neq k,k+1$ then $s\left[ s_{i}\eta %
\right] =\left( k,m\right) $ and $l\left[ s_{i}\eta \right] =l$ and we get 
\begin{gather*}
w_{s_{i}\eta }=x_{k}+(m-1)+r\left( n-1-l\right) =w_{\eta }=s_{i}\left( w_{\eta
}\right) \\
\partial _{i}\left( w_{\eta }\right) =\frac{s_{i}\left( w_{\eta }\right)
-w_{\eta }}{x_{i}-x_{i+1}}=0
\end{gather*}

and by (\ref{=SiB}) we deduce $\sigma _{i}^{+}\left( B_{\eta }\right)
=B_{s_{i}\eta }$ in this case.

For $i=k$ we have $s\left[ s_{i}\eta \right] =\left( k+1,m\right) $. If $%
\eta _{i+1}\neq \eta _{i}-1$ then we have $l\left[ s_{i}\eta \right] =l$
hence we get%
\begin{equation*}
w_{s_{i}\eta }=x_{k+1}+(m-1)+r\left( n-1-l\right) =s_{i}\left( w_{\eta
}\right) ,
\end{equation*}%
if $\eta _{i+1}\neq \eta _{i}-1$ then we have $l\left[ s_{i}\eta \right] =l+1
$ and so we get 
\begin{equation*}
w_{s_{i}\eta }=s_{i}\left( w_{\eta }\right) -r.
\end{equation*}%
In both cases $\partial _{i}\left( w_{\eta }\right) =\partial _{i}\left(
x_{i}\right) =1$ and so by (\ref{=SiB}) we get 
\begin{equation*}
\sigma _{i}^{+}\left( B_{\eta }\right) =\left\{ 
\begin{tabular}{cc}
$B_{s_{i}\eta }+r$ & if $i=k$ and $\eta _{i}-1>\eta _{i+1}$ \\ 
$B_{s_{i}\eta }$ & otherwise%
\end{tabular}%
\right. .
\end{equation*}
\end{proof}

Finally we consider the case of $C_{\eta }=B_{\eta }+rA_{\eta }$

\begin{lem}
\label{lem:Ceta}We have $\Phi ^{+}\left( C_{\eta }\right) =C_{\Phi \eta }$
and if $\eta _{i}\neq \eta _{i+1}$ then $\sigma _{i}^{+}\left( C_{\eta
}\right) =C_{s_{i}\eta }$.
\end{lem}

\begin{proof}
Since $\left( \sigma _{i}^{+}\right) ^{2}$ $=1$ it suffices to prove the $%
\sigma _{i}^{+}$-recursion for $\eta _{i}>\eta _{i+1}$. This follows from
Lemmas \ref{lem:Aeta} and \ref{lem:Beta} since the two exceptions \emph{%
cancel out} for the combination $B_{\eta }+rA_{\eta }$. The $\Phi ^{+}$%
-recursion is immediate from Lemmas \ref{lem:Aeta} and \ref{lem:Beta}.
\end{proof}

\begin{expl}
Consider the case $\eta = (1,4,1,2)$ and $i=2$. Lemma \ref{lem:Aeta} gives
$$
\sigma_2^+(A_{1,4,1,2}) = A_{1,1,4,2} - r x^{\underline{1,3,1,2}}.
$$
See Example \ref{ex:Precur}. On the other hand, Lemma \ref{lem:Beta} gives
$$
\sigma_2^+(B_{1,4,1,2}) = B_{1,1,4,2} + r x^{\underline{1,3,1,2}}.
$$
Adding these gives $\sigma_2^+( C_{1,4,1,2} ) = C_{1,1,4,2}$ as desired.
\end{expl}

We can now prove the following one-step transition formula for bar monomials

\begin{thm}
\label{thm:E}
For $\eta \neq 0$ we have  
\begin{equation}
x^{\underline{\eta }}=w_{\eta } \, x^{\underline{\eta ^{\ast }}%
}+r\sum\nolimits_{\gamma \in P\left[ \eta \right] }x^{\underline{\gamma }}.
\label{=trans}
\end{equation}
\end{thm}

\begin{proof}
The right side is, of course, the polynomial $C_{\eta }$; we set 
\begin{equation*}
Z_{\eta }=x^{\underline{\eta }}-C_{\eta }.
\end{equation*}%
By Theorem \ref{thm:bar-recur} and Lemma \ref{lem:Ceta}, we get 
\begin{equation*}
\Phi ^{+}\left( Z_{\eta }\right) =Z_{\Phi \eta }\text{ and if }\eta _{i}\neq
\eta _{i+1}\text{ then }\sigma _{i}^{+}\left( Z_{\eta }\right) =Z_{s_{i}\eta
}.
\end{equation*}%
We will prove $Z_{\eta }=0$ by induction on the size $\left\vert \eta
\right\vert $ and, for a given $\left\vert \eta \right\vert $, by \emph{%
downward} induction on the largest index $i=i\left( \eta \right) $ for which 
$\eta _{i}\neq 0$. The base case $\left( 0,\ldots ,0,1\right) $ is a
straightforward check. Now suppose we are given $\gamma \neq \left( 0,\ldots
0,1\right) $. If $i\left( \gamma \right) =n$ then we can write 
\begin{equation*}
\gamma =\Phi \eta ,\quad \eta :=\left( \gamma _{n}-1,\gamma _{1},\ldots
,\gamma _{n-1}\right) ,
\end{equation*}%
and thus $Z_{\gamma }=\Phi ^{+}\left( Z_{\eta }\right) =0$ by induction,
since $\left\vert \eta \right\vert <\left\vert \gamma \right\vert $. If $%
i\left( \gamma \right) =i<n$ then we can write%
\begin{equation*}
\gamma =s_{i}\left( \eta \right) ,\quad \eta :=\left( \gamma _{1},\ldots
,\gamma _{i-1},0,\gamma _{i},0,\ldots ,0\right) \text{;}
\end{equation*}%
and thus $Z_{\gamma }=\sigma _{i}^{+}\left( Z_{\eta }\right) =0$ by
induction, since $|\eta|=|\gamma|$ and $i\left( \eta \right) =i+1>i\left( \gamma \right) .$
\end{proof}

\begin{proof}[Proof of Theorem \protect\ref{thm:D}]
Theorem \ref{thm:D} follows by iterating Theorem \ref{thm:E}.
\end{proof}

\section{Examples, explicit formulas, and binomial coefficients} \label{Many-Examples}

We now give several detailed examples of Theorem \ref{thm:D}, leading to explicit formulas for bar monomials and interpolation polynomials. We also discuss special values of interpolation polynomials, known as binomial coefficients. These too are conjecturally positive, although this does \emph{not} follow from our formulas.
  
\subsection{Examples of Theorem \ref{thm:D}}

\label{subsec:examples}

Now we give three examples of the full computation of $x^{\underline{\gamma}}$. For brevity, when we delete the critical box without moving anything else, we record
this with a $\times$ and continue working with the same diagram. For instance, the
top middle part of Figure \ref{fig:ex104} represents the game $(1,0,4) \to
(1,0,3) \to (1,1,1) \to (0,1,1) \to (0,0,1) \to (0,0,0)$.

\vspace{0.5cm}

\begin{figure}[h]
\includegraphics[width=0.65\textwidth]{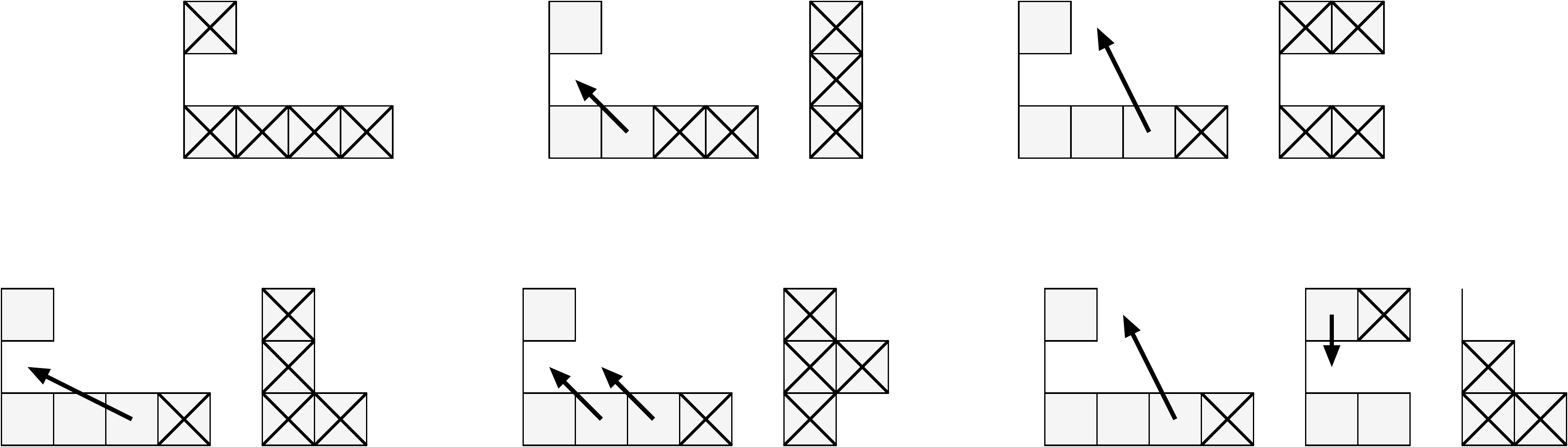}
\caption{The set $\mathcal{G}(1,0,4)$ of all games on (1,0,4)}
\label{fig:ex104}
\end{figure}

\begin{expl}
From Figure \ref{fig:ex104}, we obtain 
\begin{align*}
x^{\underline{1,0,4}} \ = &\ (x_3+3+2r) \cdot (x_3+2+2r) \cdot (x_3+1+r)
\cdot (x_1+r) \cdot x_3 \\
&+ (x_3+3+2r) \cdot r \cdot x_1 \cdot x_2 \cdot x_3 \\
&+ r \cdot (x_1+1+r) \cdot (x_3+1+r) \cdot (x_1+r) \cdot x_3 \\
&+ r \cdot (x_3+1) \cdot x_1 \cdot x_2 \cdot x_3 \\
& + r \cdot (x_2+1+r) \cdot x_1 \cdot x_2 \cdot x_3 \\
&+ r^2 \cdot (x_3 + 1 +r ) \cdot x_2 \cdot x_3 .
\end{align*}
\end{expl}

\begin{figure}[h]
\includegraphics[width=0.65\textwidth]{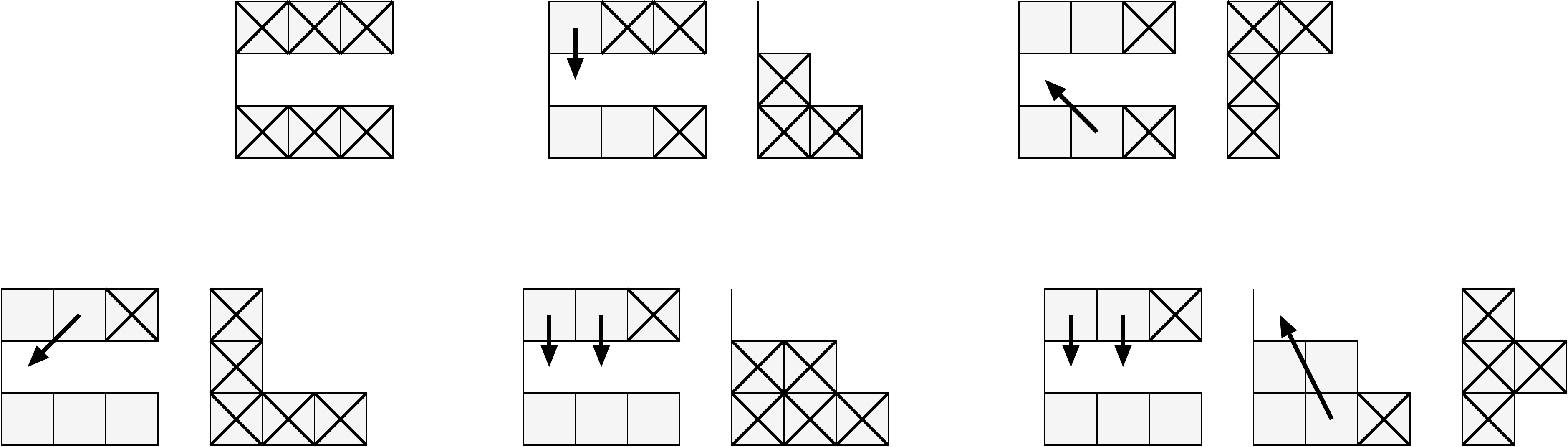}
\caption{The set $\mathcal{G}(3,0,3)$ of all games on (3,0,3)}
\label{fig:ex303}
\end{figure}

\begin{expl}
From Figure \ref{fig:ex303}, we obtain 
\begin{align*}
x^{\underline{3,0,3}} \ = &\ (x_1+2+r) \cdot (x_3+2+r) \cdot (x_1+1+r) \cdot
(x_3+1+r) \cdot (x_1+r) \cdot x_3 \\
&+(x_1+2+r) \cdot (x_3+2+r) \cdot r \cdot (x_3+1+r) \cdot x_2 \cdot x_3 \\
&+(x_1+2+r) \cdot r \cdot (x_1+1+2r) \cdot x_1 \cdot x_2 \cdot x_3 \\
&+r \cdot (x_3+2+2r) \cdot (x_3+1) \cdot x_1 \cdot x_2 \cdot x_3 \\
&+r \cdot (x_3+2+r) \cdot (x_2+1+r) \cdot (x_3+1+r) \cdot x_2 \cdot x_3 \\
&+r^2 \cdot (x_2+1+r) \cdot x_1 \cdot x_2 \cdot x_3.
\end{align*}
\end{expl}

\begin{expl}
Continuing our example from Subsection \ref{subsec:introex}, Figure \ref{fig:1241games} gives
\begin{align*}
x^{\underline{1,2,4,1}} \ =& \ (x_3+3+3r) \cdot (x_3+2+2r) \cdot (x_2+1+r) \cdot (x_3 + 1 + r) \cdot x_1 \cdot x_2 \cdot x_3 \cdot x_4\\
& \ (x_3+3+3r) \cdot r \cdot (x_2 + 1 + r) \cdot (x_4 + 1) \cdot x_1 \cdot x_2 \cdot x_3 \cdot x_4 \\
& \ r \cdot (x_1+1+r) \cdot (x_2+1+r) \cdot (x_3+1+r) \cdot x_1 \cdot x_2 \cdot x_3 \cdot x_4 \\
& \ r \cdot (x_2+1+r) \cdot (x_3+1+r) \cdot (x_4+1+r) \cdot x_1 \cdot x_2 \cdot x_3 \cdot x_4 \\
& \ r \cdot (x_4 + 2 +2r) \cdot (x_2 + 1+ r) \cdot (x_4 + 1) \cdot x_1 \cdot x_2 \cdot x_3 \cdot x_4.
\end{align*}
\end{expl}

\subsection{A combinatorial expansion for Jack interpolation polynomials}

\label{subsec:intcomb}

A fundamental result of \cite{KnS2} is that $F_\gamma^{(\alpha)}$ can be
written as a positive, weighted sum of certain ``admissible" tableaux.
Combining this result with Theorem \ref{thm:D} gives a positive,
combinatorial expansion for the Jack interpolation polynomials. We state
this result below. For the necessary combinatorial notions, we follow the
definitions and notation of \cite[sections 4-5]{KnS2}.

\begin{thm}
\label{thm:intcomb} 
Let $\gamma \in \mathbb{N}^n$. Then 
\begin{equation*}
F_\gamma^{r\delta}(x)= \sum_{T \mbox{ \tiny{\em 0-admissible}}}
d_T^0(\alpha) \sum_{G \in \mathcal{G}(\omega(T))} w(G).
\end{equation*}
Let $\gamma^{+}$ be the unique partition conjugate to $\gamma$. Then 
\begin{equation*}
J_{\gamma^{+}}^{r\delta}(x)= \sum_{T \mbox{ \tiny{\em admissible}}}
d_T(\alpha) \sum_{G \in \mathcal{G}(\omega(T))} w(G).
\end{equation*}
\end{thm}

\begin{expl}
\label{ex:F20} There are four tableaux of shape $(0,2)$ (shown below), but
only the latter two are $0$-admissible. 
\begin{equation*}
\scalebox{.6}{ \begin{tabularx}{1.61cm}{|XX|} \multicolumn{2}{|c}{ }\\
\hline \multicolumn{1}{|l|}{ \phantom{,}1\phantom{,} } &
\multicolumn{1}{l|}{ \phantom{,}1\phantom{,} } \\ \hline \end{tabularx}
\qquad \begin{tabularx}{1.61cm}{|XX|} \multicolumn{2}{|c}{ }\\ \hline
\multicolumn{1}{|l|}{ \phantom{,}1\phantom{,} } & \multicolumn{1}{l|}{
\phantom{,}2\phantom{,} } \\ \hline \end{tabularx} \qquad
\begin{tabularx}{1.61cm}{|XX|} \multicolumn{2}{|c}{ }\\ \hline
\multicolumn{1}{|l|}{ \phantom{,}2\phantom{,} } & \multicolumn{1}{l|}{
\phantom{,}1\phantom{,} } \\ \hline \end{tabularx} \qquad
\begin{tabularx}{1.61cm}{|XX|} \multicolumn{2}{|c}{ }\\ \hline
\multicolumn{1}{|l|}{ \phantom{,}2\phantom{,} } & \multicolumn{1}{l|}{
\phantom{,}2\phantom{,} } \\ \hline \end{tabularx} }
\end{equation*}
Hence 
\begin{align*}
F^{r\delta}_{(0,2)} &= \textstyle (\frac{2}{r}+2) \, x^{\underline{1,1}} 
+ \textstyle (\frac{2}{r}+2)(\frac{1}{r}+1) \, x^{\underline{0,2}} \\
&= \textstyle (\frac{2}{r}+2) \, x_1 x_2
+ \textstyle (\frac{2}{r}+2)(\frac{1}{r}+1) (x_2+1+r)x_2. \\
\end{align*}
On the other hand, all four tableaux of shape (2,0) are $0$-admissible. We
get 
\begin{align*}
F^{r\delta}_{(2,0)} &= \textstyle (\frac{2}{r}+1)(\frac{1}{r}+1) \, x^{\underline{2,0}} 
+ \left( \textstyle (\frac{2}{r}+1) + 1 \right) x^{\underline{1,1}} 
+ \textstyle (\frac{1}{r}+1) \, x^{\underline{0,2}} \\
&= \textstyle (\frac{2}{r}+1)(\frac{1}{r}+1) \Big( (x_1+1+r)(x_1+r) + r(x_2) \Big) \\
& \qquad + \left( \textstyle (\frac{2}{r}+1) + 1 \right) x_1 x_2 + 
\textstyle (\frac{1}{r}+1) (x_2+1+r) x_2.
\end{align*}
and 
\begin{align*}
J^{r\delta}_{(2,0)} &= \textstyle (\frac{1}{r}+1) \, x^{\underline{2,0}} + 2 \, x^{%
\underline{1,1}} + \textstyle (\frac{1}{r}+1) \, x^{\underline{0,2}} \\
&= \textstyle (\frac{1}{r}+1) \Big( (x_1+1+r)(x_1+r) + r(x_2) \Big)+ 2 \, x_1
x_2 + \textstyle (\frac{1}{r}+1) (x_2+1+r) x_2.
\end{align*}
\end{expl}

\begin{figure}[h]
\includegraphics[width=3in]{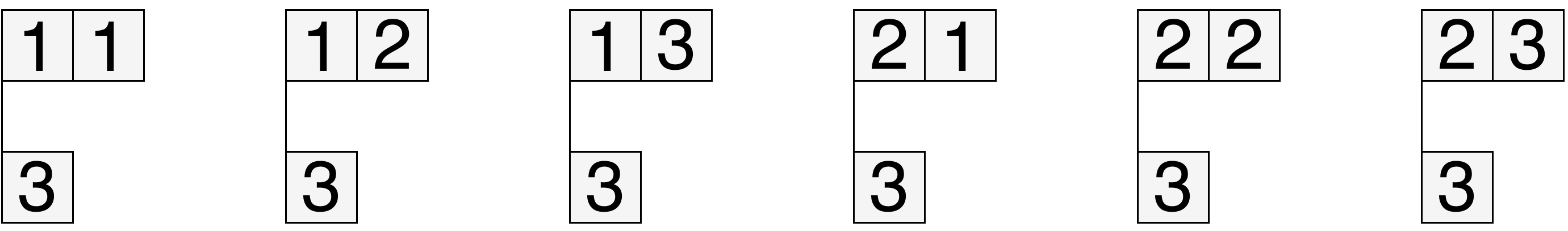}
\caption{All $0$-admissible tableau of shape $(2,0,1)$.}
\label{fig:adm201}
\end{figure}

\begin{expl}
There are six $0$-admissible tableaux of shape $(2,0,1)$. They are given in
Figure \ref{fig:adm201}. The weights $\omega$ of these tableaux are $(2,0,1)$%
, $(1,1,1)$, $(1,0,2)$, $(1,1,1)$, $(0,2,1)$, and $(0,1,2)$, respectively.
Hence 
\begin{align*}
F^{r\delta}_{(2,0,1)} 
&= \textstyle (\frac{2}{r}+2)(\frac{1}{r}+1)(\frac{1}{r}+2) \, x^{\underline{(2,0,1)}} \\
&+ \textstyle (\frac{2}{r}+2)(\frac{1}{r}+2) \, x^{\underline{(1,1,1)}} \\
&+ \textstyle (\frac{2}{r}+2)(\frac{1}{r}+2) \, x^{\underline{(1,0,2)}} \\
&+ \textstyle (\frac{1}{r}+2) \, x^{\underline{(1,1,1)}} \\
&+ \textstyle (\frac{1}{r}+1)(\frac{1}{r}+2) \, x^{\underline{(0,2,1)}} \\
&+ \textstyle (\frac{1}{r}+2) \, x^{\underline{(0,1,2)}}.
\end{align*}
To further expand, we need to look at games. Notice that among all the games
of shapes $(2,0,1)$, $(1,1,1)$, $(1,0,2)$, $(0,2,1)$, and $(0,1,2)$, there
is only one game with a non-trivial move: $(2,0,1) \to (0,1,1) \to (0,0,1)
\to (0,0,0)$. Hence we get the following expansion: 
\begin{align*}
F^{r\delta}_{(2,0,1)}  
&= \textstyle (\frac{2}{r}+2)(\frac{1}{r}+1)(\frac{1}{r}+2) 
\Big( (x_1+1+2r)(x_1+r)x_3 + r x_2 x_3 \Big) \\
&+ \textstyle (\frac{2}{r}+2)(\frac{1}{r}+2) \, x_1 x_2 x_3 \\
&+ \textstyle (\frac{2}{r}+2)(\frac{1}{r}+2) (x_3 + 1+ r)(x_1+r)x_3 \\
&+ \textstyle (\frac{1}{r}+2) \, x_1 x_2 x_3 \\
&+ \textstyle (\frac{1}{r}+1)(\frac{1}{r}+2) (x_2+1+2r) x_2 x_3 \\
&+ \textstyle (\frac{1}{r}+2) (x_3+1+r) x_2 x_3. \\
\end{align*}
\end{expl}

\subsection{Vanishing properties}
\label{subsec:vanprop}

By definition, the bar monomials have lower vanishing properties. For instance, $\x{3,0}$ vanishes at $\overline{(1,1)} = (-1-r,-1)$. However, this does not happen game by game. Combinatorially it is not clear why it happens at all.

Furthermore, when the interpolation Jack polynomials are evaluated at shapes that are larger in the containment order, it seems that we get positive Laurent polynomials in $r$ (up to an overall sign). These polynomials are called \textit{binomial coefficients} \cite{Bi,Lass,OO}. But this is not true at the level of bar monomials (much less at the level of games), and again the combinatorics is obscure.

We give examples to illustrate the two phenomena.

\begin{expl} Vanishing of $\x{3,0}$ at $\overline{(1,1)} = (-1-r,-1)$
$$
\x{3,0} = (x_1+2+r)(x_1+1+r)(x_1+r) + (x_1+2+r)rx_2 + r(x_2+1+r)x_2 + rx_1x_2
$$
and at $\overline{(1,1)} = (-1-r,-1)$ we get 

\begin{align*}
(x_1+2+r)(x_1+1+r)(x_1+r) \quad &\rightarrow \quad 0 \phantom{(x_1+2+r)(x_1+1+r)(x_1+r)} \\
(x_1+2+r)rx_2 \quad &\rightarrow \quad -r \\
r(x_2+1+r)x_2 \quad &\rightarrow \quad -r^2 \\
rx_1x_2 \quad &\rightarrow \quad r^2+r
\end{align*}
\end{expl}

\begin{expl} Positivity of $F^{r\delta}_{(3,1)}$ at $\overline{(3,4)} = (-3,-4-r)$
\begin{align*}
F^{r\delta}_{(3,1)} &= \textstyle (\frac{3}{r}+2)(\frac{2}{r}+1)(\frac{1}{r}+1)^2 \, \x{3,1} + \textstyle (\frac{3}{r}+2)(\frac{1}{r}+1) \, \x{2,2}\\
& \ + \textstyle (\frac{3}{r}+2)(\frac{2}{r}+1)(\frac{1}{r}+1) \, \x{2,2} + \textstyle (\frac{3}{r}+2)(\frac{1}{r}+1)^2 \, \x{1,3}\\
&= \textstyle (\frac{3}{r}+2)(\frac{2}{r}+1)(\frac{1}{r}+1)^2 \Big( (x_1+2+r)(x_1+1+r) x_1 x_2 + r(x_2+1) x_1 x_2 \Big) \\
& \ + \textstyle (\frac{3}{r}+2)(\frac{1}{r}+1) \, (x_1+1)(x_2+1)x_1x_2\\
& \ + \textstyle (\frac{3}{r}+2)(\frac{2}{r}+1)(\frac{1}{r}+1) \, (x_1+1)(x_2+1)x_1x_2\\
& \ + \textstyle (\frac{3}{r}+2)(\frac{1}{r}+1)^2 \, (x_2+2+r)(x_2+1)x_1x_2\\
\end{align*}
Evaluating this at $\overline{(3,4)} = (-3,-4-r)$ gives\\
\begin{align*}
\frac{144}{r^4}+ \ \frac{60}{r^3} \ - \, \, \frac{834}{r^2}-\frac{1530}{r}-1074-330r-36r^2&\\[10pt]
+\, \, \frac{432}{r^3} \, +\frac{1188}{r^2}+\frac{1230}{r}+\, 600 \, +138r+12r^2&\\[10pt]
+\, \, \frac{216}{r^2} \ +\ \frac{486}{r} \,+\, 372 \,+114r+12r^2&\\[10pt]
+ \, \, \frac{216}{r^3} \, +\, \frac{702}{r^2} \ +\ \frac{858}{r} \, +\, 486 \,+126r+12r^2&\\[10pt]
= \frac{144}{r^4}+\, \frac{708}{r^3} \, +\frac{1272}{r^2}+\frac{1044}{r}+\, 384 \,+48r\phantom{5+12r^2}&\\
\end{align*}
\end{expl}
Currently there is no manifestly positive combinatorial formula for the binomial coefficients, except in some small cases \cite{KnS1,Sa4,NS}.
Understanding the lower vanishing properties of the bar monomials from a combinatorial perspective may shed more light on the binomial coefficient problem.

\bibliography{barmon}
\bibliographystyle{abbrv} 

\end{document}